\documentclass[12pt]{amsart}
\usepackage{amssymb}
\usepackage{amsfonts}
\usepackage{latexsym}
\usepackage{amscd}
\usepackage[mathscr]{euscript}

\usepackage[active]{srcltx}

\usepackage{enumerate}

\usepackage[utf8]{inputenc}
\usepackage[all,cmtip,2cell]{xy}

\usepackage[colorlinks  = true,
            citecolor   = blue,
            linkcolor   = blue,
            bookmarks   = true,
            linktocpage = true]{hyperref}

\usepackage{tikz}
\usetikzlibrary{matrix}

\numberwithin{equation}{section}

\theoremstyle{plain}
\newtheorem{lemma}{Lemma}[section]

\newtheorem{theorem}[lemma]{Theorem}
\newtheorem{corollary}[lemma]{Corollary}
\newtheorem{proposition}[lemma]{Proposition}
\newtheorem{definition}[lemma]{Definition}
\newtheorem*{proposition*}{Proposition}
\newtheorem*{theorem*}{Theorem}
\newtheorem*{definition*}{Definition}
\newtheorem*{claim*}{Claim}
\newtheorem*{notation*}{Notation}
\newtheorem*{notation}{Notation}

\newtheorem{point}[lemma]{}

\newtheorem{remark}[lemma]{Remark}
\newtheorem{example}[lemma]{Example}

\newcommand{\Gg}{\mathcal{G}}
\newcommand{\Hh}{\mathcal{H}}

\newcommand{\Oo}{\mathcal{O}}

\newcommand{\G}{\mathcal{G}}

\newcommand{\go}{\Gg^{(0)}}

\newcommand{\lsp}{\operatorname{span}}

\newcommand{\ZZ}{\mathbb{Z}}

\newcommand{\N}{{\mathbb{N}}}

\newcommand{\Z}{{\mathbb{Z}}}

\newcommand{\ol}{\overline}
\newcommand{\uloopr}[1]{\ar@'{@+{[0,0]+(-4,5)}@+{[0,0]+(0,10)}@+{[0,0] +(4,5)}}^{#1}}
\newcommand{\uloopd}[1]{\ar@'{@+{[0,0]+(5,4)}@+{[0,0]+(10,0)}@+{[0,0]+ (5,-4)}}^{#1}}
\newcommand{\dloopr}[1]{\ar@'{@+{[0,0]+(-4,-5)}@+{[0,0]+(0,-10)}@+{[0, 0]+(4,-5)}}_{#1}}
\newcommand{\dloopd}[1]{\ar@'{@+{[0,0]+(-5,4)}@+{[0,0]+(-10,0)}@+{[0,0 ]+(-5,-4)}}_{#1}}

\newcommand{\typ}{\mbox{\rm typ}}
\newcommand{\Typ}{\mbox{\rm Typ}}

\newcommand{\luloop}[1]{\ar@'{@+{[0,0]+(-8,2)}@+{[0,0]+(-10,10)}@+{[0, 0]+(2,2)}}^{#1}}

\newcommand{\Ifree}{I_{\mathrm{free}}}
\newcommand{\Ireg}{I_{\mathrm{reg}}}

\renewcommand{\_}{\ \cdot\ }

\newdimen \boxht
  \def \sumTwoIndices #1#2{\sum _{\buildrel {\scriptstyle #1}\over {\vrule height 8pt width 0pt #2}}}

\begin{document}
\title[Type Semigroup]{The groupoids of adaptable separated graphs and their type semigroups}

\author{Pere Ara}\author{Joan Bosa}
\address{Departament de Matem\`atiques, Universitat Aut\`onoma de Barcelona,
08193 Bellaterra (Barcelona), Spain, and Barcelona Graduate School of Mathematics (BGSMath).} \email{para@mat.uab.cat}\email{jbosa@mat.uab.cat}

\author{Enrique Pardo}
\address{Departamento de Matem\'aticas, Facultad de Ciencias, Universidad de C\'adiz,
Campus de Puerto Real, 11510 Puerto Real (C\'adiz),
Spain.}\email{enrique.pardo@uca.es}\urladdr{https://sites.google.com/a/gm.uca.es/enrique-pardo-s-home-page/}
\author{Aidan Sims}
\address{School of Mathematics and Applied Statistics, University of Wollongong, Wollongong NSW 2522, Australia.} \email{asims@uow.edu.au}

\thanks{The first, second and third authors were partially supported by the DGI-MINECO and European Regional Development Fund, jointly, through the grant MTM2017-83487-P.
The first and second authors acknowledge support from the Spanish Ministry of Economy and Competitiveness, through the María de Maeztu Programme for Units of Excellence in R$\&$D (MDM-2014-0445)
and from the Generalitat de Catalunya through the grant 2017-SGR-1725.
The third author was partially supported by PAI III grant FQM-298 of the Junta de Andaluc\'{\i}a. The fourth author was partially supported by the Australian Research Council grant DP150101595.}
\subjclass[2010]{Primary 16D70, Secondary 06F20, 19K14, 20K20, 46L05, 46L55}
\keywords{Steinberg algebra, Refinement monoid, Type semigroup.}
\date{\today}
\dedicatory{Dedicated to the memory of Emili Ara Aymami}
\maketitle

\begin{abstract}
Given an adaptable separated graph, we construct an associated groupoid and explore its type semigroup.
Specifically, we first attach to each adaptable separated graph a corresponding semigroup, which we prove is an $E^*$-unitary inverse semigroup.
As a consequence, the tight groupoid of this semigroup is a Hausdorff \'etale groupoid. We show that this groupoid is always amenable,
and that the type semigroups of groupoids obtained from adaptable separated graphs in this way include all finitely generated conical refinement monoids.
The first three named authors will utilize this construction in forthcoming work to solve the Realization Problem for von Neumann regular rings, in the finitely generated case.
\end{abstract}

\section*{Introduction.}\label{Sect:Intro}
There has been significant recent interest in the structure of the type semigroup of an ample
groupoid \cite{BL, PSS, RS} and its influence in determining properties of the associated
(reduced) groupoid $C^*$-algebra. In particular the stably finite versus purely infinite dichotomy
has been established, under mild hypotheses in the above-mentioned papers, using the type
semigroup. Additionally, the enveloping group of the type semigroup of an ample groupoid $\mathcal
G$ is precisely the homology group $H_0(\mathcal G)$ of the groupoid, and so the induced map $\Typ
(\mathcal G) \to \mathcal V (C^*_r(\mathcal G))$ from $\Typ (\mathcal G)$ to the Murray-von
Neumann semigroup $\mathcal V (C^*_r(\mathcal G))$ can be viewed as a non-stable precursor of the
natural map $H_0(\mathcal G) \to K_0(C^*_r(\mathcal G))$ appearing in the statement of Matui's
Conjecture (recently resolved in the negative by Scarparo \cite{Scarparo}). In particular it is of
great interest to determine the family of ample groupoids such that the map $\Typ (\mathcal G) \to
\mathcal V (C^*_r(\mathcal G))$ and its stabilized sibling $H_0(\mathcal G) \to K_0(C^*_r(\mathcal
G))$ are injective (see \cite{Matui, NO, Ortega, FKPS}).

In this paper, we introduce a new class of ample Hausdorff groupoids associated to separated
graphs, and describe concrete representations of their Steinberg algebras and $C^*$-algebras. This
new class of groupoids generalizes the well-known graph groupoids of row-finite graphs
\cite{KPRR}, and are designed to have a more general type semigroup than those. Specifically, for
a row-finite graph $E$, our methods demonstrate that the type semigroup of the associated graph
groupoid is canonically isomorphic to the graph monoid $M(E)$ introduced in \cite{AMFP}; that is,
the commutative monoid with generators $\{ a_v : v\in E^0 \}$, and defining relations
$a_v=\sum_{e\in s^{-1}(v)} a_{r(e)}$ whenever $v$ is not a sink in $E$. The graph monoids
$M(E)$, as well as all the monoids $\Typ (\mathcal G)$ of an ample groupoid $\mathcal G$, are {\it
refinement monoids}, in the sense that they satisfy the Riesz refinement property. However, the
class of graph monoids does not even cover the very natural class of all finitely generated
conical refinement monoids. Indeed it was already shown in \cite{APW08} that the finitely
generated refinement monoid $\langle p,q_1,q_2: p+q_1=p=p+q_2 \rangle$ is not a graph monoid; that is, it is not
isomorphic to $M(E)$ for any row-finite graph $E$. In this paper we show that this problem can be
solved by considering a special class of separated graphs, which we call \emph{adaptable}
separated graphs.

Our work is also motivated by the {\it realization problem for von Neumann regular rings} (see for
instance \cite{Areal,directsum}). Using the construction in this paper, the three first-named
authors show in \cite{ABP19b} that every finitely generated conical refinement monoid arises
as the monoid $\mathcal{V}( A_K(\mathcal G)\Sigma^{-1})$ associated to a von Neumann regular ring of
the form $A_K(\mathcal G)\Sigma^{-1}$ for a suitable universal localization of the Steinberg
algebra $A_K(\mathcal G)$ \cite{CFST,
Steinberg}, where $\mathcal G$ belongs to the class of groupoids constructed here
and $K$ is an arbitrary field. Steinberg algebras of ample groupoids have received quite a bit of attention
in the last few years, see for instance the survey papers \cite{CH, Rigby}.

Recall from \cite{AG12} that a separated graph is a pair $(E,C)$ where $E$ is a directed graph and
$C=\bigsqcup_{v\in E^0} C_v$ is a partition of $E^1$ which is finer than the partition induced by
the source function $s\colon E^1\to E^0$. Given a finitely separated graph, that is, a separated
graph such that all the sets in the partition $C$ are finite, the monoid $M(E,C)$ is defined in
\cite{AG12} as the commutative monoid with generators $\{ a_v : v\in E^0 \}$ and defining
relations given by the equations $a_v= \sum_{x\in X} a_{r(x)}$ for $v \in E^0$ and $X\in C_v$.
Unlike graph monoids, $M(E,C)$ is not always a refinement monoid, but sufficient conditions under
which $M(E, C)$ is a refinement monoid were identified in \cite[Section 5]{AG12}. The key
observation for the present paper is that a restricted class of separated graphs, which we have
termed {\it adaptable separated graphs}, is enough for our purposes. In fact, for this family of
separated graphs $(E,C)$, the refinement property for $M(E,C)$ does hold, and the monoids $M(E,C)$
belong to a especially well-behaved class of refinement monoids, the primely generated refinement
monoids. Moreover any finitely generated conical refinement monoid arises as $M(E,C)$ for some
adaptable separated graph $(E,C)$; see \cite{ABP19} for all these results, which we have also
summarized in Section~\ref{Sect:Prelim}.

Having identified the suitable combinatorial object, namely adaptable separated graphs, we
associate to each such separated graph a suitable inverse semigroup. Our guiding model is the
graph inverse semigroup defined in \cite{JL}, but difficulties arise in trying to modify the
definition of the graph inverse semigroup from \cite{JL} to our context. These arise from the fact
that the algebra relations naturally attached to a separated graph \cite{AG12} are not tame; that
is, they do not force commutation of range projections of the generating partial isometries, which
precludes their generating an inverse semigroup. This problem was solved in \cite{AE} by forming a
universal quotient, called the tame algebra of the separated graph, in which the range projections
of the generating partial isometries do commute. Unfortunately, though, this introduced a new
difficulty problem: in general, passing from the universal $C^*$-algebra or algebra of a separated
graph to its tame quotient as in \cite{AE} induces a nontrivial map from the monoid $M(E, C)$ to
the monoid $\mathcal{VO}(E,C)$ or $\mathcal{V}L^{{\rm ab}}_K(E,C)$; so we can no longer be sure
that we are representing the desired monoid $M(E,C)$.

Our solution to this problem is to introduce a set of auxiliary variables $\{ t_i^{v}: i\in \N,
v\in E^0 \}$ which are designed to tame the natural relations associated to the separated graph
$(E,C)$, without altering the associated monoid. Indeed the starting point of our paper is the
construction of an inverse semigroup $S(E,C)$ for any adaptable separated graph $(E,C)$, which
makes use of these auxiliary variables $t_i^v$.  The crucial semilattice of idempotents $\mathcal
E$ is described solely in terms of paths and monomials in the separated graph, and we are able to
show that $S(E,C)$ is an $E^*$-unitary inverse semigroup, and consequently the tight groupoid
$\mathcal G_{tight} (S(E,C))$ is an ample Hausdorff groupoid (see
Section~\ref{sect:def-inverse-semigroupS}). Following the by-now well-trodden path described in
\cite{ExelBraz}, we then build the space of tight filters $\widehat{\mathcal E}_{tight}$
associated to $S(E,C)$, and we relate it to the space of infinite paths in our separated graph
$(E,C)$. The desired groupoid is obtained as the groupoid $\mathcal G_{tight}(S(E,C))$ of germs of
the canonical action of $S(E,C)$ on the space $\widehat{\mathcal E}_{tight}$ of tight filters. Our
main result is Theorem~\ref{thm:monoids-agreement}: for any adaptable separated graph $(E,C)$, the
type semigroup of the ample Hausdorff groupoid $\mathcal G_{tight} (S(E,C))$ is canonically
isomorphic to the graph monoid $M(E,C)$. In particular, every finitely generated conical
refinement monoid arises as the type semigroup of a groupoid in our class.

We show in addition that the groupoids $\mathcal G := \mathcal G_{tight}(S(E,C))$ are all amenable
(Proposition~\ref{prop:amenable}), and that the corresponding Steinberg algebras $A_K(\mathcal G)$
and $C^*$-algebras $C^*(\mathcal G ) = \mathcal C^*_r(\mathcal G)$ can be described as universal
objects for appropriate generators and relations
(Theorem~\ref{thm:Steinbergalgebra-and-our-algebra} and
Corollary~\ref{cor:Groupoid-C*-algebra-and-our-algebra}). We also obtain a very concrete
description of the groupoid $\mathcal G_{tight}(S(E,C))$ (see Theorem~\ref{thm:iso-groupoids}),
similar to the one given in \cite{KPRR} for usual graphs, which makes computations much more
tractable.

\medskip

We briefly outline the contents of this paper. In Section~\ref{Sect:Prelim}, we record the basic
definitions and results we will need throughout the rest of the paper. In particular we introduce
the key notion of an {\it adaptable separated graph}. In
Section~\ref{sect:def-inverse-semigroupS}, we define a natural inverse semigroup $S(E,C)$
associated to an adaptable separated graph $(E,C)$. We provide two equivalent definitions of the
semigroup and analyse its semilattice of idempotents $\mathcal E$. We also show that $S(E,C)$ is
an $E^*$-unitary inverse semigroup (Proposition~\ref{prop:E-star-unitary-general}). We analyse the
structure of the spaces of filters, ultrafilters, and tight filters on $\mathcal E$ in
Section~\ref{sect:filters}. We show in Theorem~\ref{thm:tightfilters} that the space of
ultrafilters on $\mathcal E$  coincides with the space of tight filters on $\mathcal E$. In
Section~\ref{Sect:Algebra-Semigroup}, we show that the natural $K$-algebra $\mathcal{S}_K(E,C)$,
associated to an adaptable separated graph $(E,C)$ is isomorphic to the Steinberg algebra
$A_K(\mathcal{G}_{\text{tight}}(S(E,C)))$ of the groupoid $\mathcal{G}_{\text{tight}}(S(E,C))$ of
germs of the canonical action of $S(E,C)$ on the space of tight filters on $\mathcal E$. In
Section~\ref{sect:description-groupoid}, we provide an equivalent, more concrete picture of
$\mathcal{G}_{\text{tight}}(S(E,C))$. This is used in Section~\ref{Sect:Amenability} to show that
$\mathcal{G}_{\text{tight}}(S(E,C))$ is an amenable groupoid. We finish by showing in
Section~\ref{Sect:TypeSemigroup} that, for any adaptable separated graph $(E,C)$, the type
semigroup $\Typ(\mathcal{G}_{\text{tight}}(S(E,C)))$ of the groupoid $\mathcal
G_{\text{tight}}(S(E,C))$ is naturally isomorphic to the monoid $M(E,C)$ associated to $(E,C)$. We
deduce from the results of \cite{ABP19} that every finitely generated conical refinement monoid
arises as the type semigroup of the groupoid associated to an adaptable separated graph.

\section{Preliminaries.}\label{Sect:Prelim}

In this section, we will recall the basic definitions and results needed to follow the paper.


\subsection{Basics on commutative monoids.}
We will denote by $\N$ the semigroup of positive integers, and by $\Z^+$ the monoid of non-negative integers.

Given a commutative monoid $M$, we set $M^*:=M\setminus\{0\}$. We say that $M$ is {\it conical} if $M^*$ is a semigroup, that is, if, for all $x$, $y$ in $M$, $x+y=0$ only when $x=y=0$.

We say that  $M$ is
a {\it refinement monoid} if, for all $a$, $b$, $c$, $d$ in
$M$ such that $a+b=c+d$, there exist $w$, $x$, $y$, $z$ in $M$ such
that $a=w+x$, $b=y+z$,
$c=w+y$ and $d=x+z$.

A basic example of a refinement monoid is the monoid $M(E)$ associated to a countable row-finite graph $E$ \cite[Proposition 4.4]{AMFP}.

If $x, y\in M$, we write $x\leq y$
if there exists $z\in M$  such that $x+z = y$.
Note that $\le$ is a translation-invariant pre-order on $M$, called the {\it algebraic pre-order} of $M$. All inequalities in commutative monoids will be with respect to this pre-order.
An element $p$ in a monoid $M$ is a {\it prime element} if $p$ is not invertible in $M$, and, whenever
$p\leq a+b$ for $a,b\in M$, then either $p\leq a$ or $p\leq b$. The monoid $M$ is {\it primely generated} if every non-invertible element of $M$
can be written as a sum of prime elements.

An element $x\in M$ is {\it regular} if $2x\leq x$. An element $x\in M$ is an {\it idempotent} if $2x= x$. An element
$x\in M$ is {\it free} if $nx\leq mx$ implies $n\leq m$, for $n,m\in \N$. By \cite[Theorem 4.5]{Brook}, any element
of a primely generated refinement monoid is either free or regular.

\subsection{Adaptable Separated Graphs}\label{Section1}
In \cite{AP16} the first and third-named authors characterized the primely generated conical refinement monoids in terms of the so-called $I$-systems.
Using this theory,
the first, second and third-named authors obtained in \cite{ABP19} a combinatorial model for all finitely generated conical refinement monoids.
The basic ingredient in this combinatorial description is the theory of separated graphs \cite{AG12}.
(Note that ordinary graphs are not sufficient for this purpose, see \cite{AP17, APW08}.)


\begin{definition}[{\cite[Definitions 2.1 and 4.1]{AG12}}]\label{defsepgraph}
    {\rm A \emph{separated graph} is a pair $(E,C)$ where $E$ is a directed graph,  $C=\bigsqcup
    _{v\in E^ 0} C_v$, and
    $C_v$ is a partition of $s^{-1}(v)$ (into pairwise disjoint nonempty
    subsets) for every vertex $v$. If $v$ is a sink, we take $C_v$
    to be the empty family of subsets of $s^{-1}(v)$.

    If all the sets in $C$ are finite, we shall say that $(E,C)$ is a \emph{finitely separated} graph.

    Given a finitely separated graph $(E,C)$, we define the monoid of the separated graph $(E,C)$ to
    be the commutative monoid given by generators and relations as
    follows:
    $$M(E,C)=\Big{\langle} a_v \,\, \, (v\in E^0) \, : a_v=\sum _{\{ e\in X\}}a_{r(e)} \text{ for every } X\in C_v, v\in E^0\Big{\rangle} .$$ }
\end{definition}
We recall some basic graph-theoretic notions that we will need along the sequel. For further background, see, for
example, \cite{AAS}.
\begin{definition}
    {\rm Given a directed graph $E=(E^0,E^1, s,r)$:
        \begin{enumerate} 
        \item We define a pre-order on $E^0$ (the path-way pre-order) by $v\le w$ if and only if there is a directed path $\gamma $ in $E$ with $s(\gamma ) = w$ and $r(\gamma ) =v$.
        We denote by $\sim$ the equivalence relation associated to the pre-order $\le $ on $E^0$, that is, $v\sim w$ if and only if there are directed
            paths both from $v$ to $w$ and from $w$ to $v$. The equivalence classes with respect to $\sim$ are called the {\it transitive components} of $E$. 
         \item We will denote by $(I,\leq)$ the poset arising as the antisymmetrization of $(E^0,\leq )$, so that the elements of $I$ are the different transitive components and,
         denoting by $[v]$ the class of $v\in E^0$ in $I$,
                       we have $[v]\leq [w]$ if and only if $v\leq w$.
            \item We say that $E$ is {\it transitive} if every two vertices of $E^0$ are connected through a directed path, i.e., if $I$ is a singleton.
            \end{enumerate}
    }
\end{definition}

Before giving the formal definition (which is quite technical) of an adaptable separated graph,
which will be the combinatorial model we use throughout the paper, we give an informal description
of the basic idea. An adaptable separated graph is a separated graph $(E, C)$ such that $E$ has only
finitely many transitive components, partitioned into two types: the \emph{free} components, which
have just one vertex, and the \emph{regular} components, which can have multiple vertices and must
be row-finite graphs. The minimal free components $\{v\}$ must be sinks in $E$ in the sense that
$s^{-1}(v) = \emptyset$. For each non-minimal free component $\{v\}$ each set $X_i$ in the
partition $C_v$ of $s^{-1}(v)$ contains one loop $\alpha_i$ pointing from $v$ to itself, and then
a finite number of edges $\beta_{i,j}$ that point to transitive components different from $\{v\}$.
The regular components must be ``conventional" graphs in the sense that at every vertex $v$ in a
regular component, the partition $C_v$ is the trivial partition $\{s^{-1}(v)\}$, and they must
have the property that every vertex in the component emits at least two edges in that component.

The idea behind the use of the word ``adaptable" is that given a pre-specified partitioned and
partially ordered set $(\Ifree \sqcup \Ireg, \le)$, the separated graph $(E, C)$ is ``adapted" to
$(I, \le)$, if it is an adaptable separated graph and there is an order isomorphism between the
ordered set of transitive components of $(E, C)$ and the partial order $(\Ifree \sqcup \Ireg,
\le)$ that carries the free components to $\Ifree$ and the regular components to $\Ireg$.

We now present the formal definition.

\begin{definition}
    \label{def:adaptable-sepgraphs}
    {\rm Let $(E,C)$ be a finitely separated graph and let $(I,\le )$ be the antisymmetrization of $(E^0,\leq)$. We say that $(E,C)$ is {\it adaptable} if $I$ is finite,
    and there exist a partition $I=\Ifree \sqcup \Ireg $, and a family of subgraphs $\{ E_p \}_{p\in I}$ of $E$ such that the following conditions are satisfied:
        \begin{enumerate}
            \item $E^0=\bigsqcup_{p\in I} E_p^0$, where $E_p$ is a transitive row-finite graph
                if $p\in \Ireg$ and $E_p^0= \{ v^p \}$ is a single vertex if $p\in \Ifree$.
            \item For $p\in \Ireg$ and $w\in E_p^0$, we have that $|C_w|= 1$ and
                $|s_{E_p}^{-1}(w)| \ge 2$. Moreover, all edges departing from $w$ either belong to the graph $E_p$ or connect $w$ to a vertex $u\in E_q^0$,
            with $q<p$ in $I$.
            \item For $p\in \Ifree$, we have that $s^{-1}(v^p) = \emptyset $ if and only if $p$
                is minimal in $I$. If $p$ is not minimal, then $C_{v^p}$ partitions $s^{-1}(v)$ into
                finitely many sets $C_{v^p}=\{X^{(p)}_1,\dots ,X^{(p)}_{k(p)}\}$. For each $i
                \le k(p)$, the set $X^{(p)}_i$ is of the form
            $$X^{(p)}_i = \{ \alpha (p,i) ,\beta (p,i,1),\beta (p,i,2),\dots , \beta(p, i, g(p,i)) \}$$
            for some $g(p,i) \ge 1$,  where $\alpha (p,i)$ is a loop, i.e., $s(\alpha(p,i))= r(\alpha (p,i)) = v^p$, and $r(\beta (p,i,t))\in E^0_q$ for $q<p$ in $I$.
           We have $E_p^1 = \{\alpha(p,1), \dots, \alpha (p, k(p))\}$.
           \end{enumerate}
        The edges connecting a vertex $v\in E_p^0$ to a vertex $w\in E_q^0$ with $q<p$ in $I$ will be called {\it connectors}.} \qed
\end{definition}

Observe that since the order on $(I, \le)$ is induced by the path-way pre-order on $E^0$, if $e
\in E^1$ and $s(e) \in E_p^0$ with $p \in \Ireg$, then $r(e)$ either belongs to $E_p^0$ (which happens if and only if $e\in E^1_p$), or
belongs to $E_q^0$ for some $q$ with $q < p$. Likewise, for $p \in \Ifree$, each edge of the form
$\beta(p,i,j)$ must point to some component $E_q$ with $q < p$. 

This is the main theorem connecting graphs and monoids.

\begin{theorem}\cite{ABP19}
\label{thm:maingraphs-monoids}
    The following two statements hold:
    \begin{enumerate}
        \item If $(E,C)$  is an adaptable separated graph, then $M(E,C)$ is a primely generated conical refinement monoid.
        \item For any finitely generated conical refinement monoid $M$, there exists an adaptable separated graph $(E,C)$ such that $M\cong M(E,C)$.
    \end{enumerate}
\end{theorem}

In particular, it is shown in \cite{ABP19} that, for an adaptable separated graph $(E,C)$, all the
elements $a_v$, for $v\in E^0$, are prime elements of the monoid $M(E,C)$, and that $a_v$ is free
(respectively, regular) in $M(E,C)$ if and only if $[v]\in \Ifree$ (respectively, $[v]\in \Ireg$).
We often refer to the elements of $\Ifree$ as {\it free primes} and to the elements of $\Ireg$ as
{\it regular primes}.


\subsection{Groupoids and Steinberg algebras}




Recall that a groupoid is a small category in which every morphism
is an isomorphism (see \cite{Renault} for further details). Given a
groupoid ${\mathcal G}$, we will always denote its \emph {unit
space} by $\go $, the set of \emph{composable pairs} by ${\mathcal
G}^{(2)}$, and its \emph {source} and \emph {range} maps by $s$ and
$r$, respectively. A \emph {bisection} in ${\mathcal G}$ is a subset
$U\subseteq {\mathcal G}$ such that the restrictions of $r$ and $s$
to $U$ are both injective.

A topological groupoid ${\mathcal G}$ is said to be \emph {\'etale} if $\go $ is locally compact
and Hausdorff in the relative topology, and its range map (equivalently, its source map) is a
local homeomorphism from ${\mathcal G}$ to $\go $. It is easy to see that the topology of an
\'etale groupoid admits a basis of open bisections. In an \'etale groupoid, $\go $ is open in
${\mathcal G}$. If, in addition,  ${\mathcal G}$ is Hausdorff, then $\go $ is also closed in
${\mathcal G}$.  An \'etale groupoid ${\mathcal G}$ is \emph {ample} if $\mathcal G$ admits a
basis of open compact bisections.

\begin {definition}[\cite {Steinberg}, \cite {CFST}]
\label{def:Steinbergalgebra}
{\rm Given an ample groupoid ${\mathcal G}$, and a field with involution
$(K,*)$, the \emph{Steinberg algebra} associated to ${\mathcal G}$
is defined to be the $*$-algebra over $K$
$$A_K(\mathcal G)= \lsp\{1_B : B \text{ is an open compact bisection }\}$$ with the \emph{convolution product}
  $$
  (fg)(\gamma ) = \sumTwoIndices {(\gamma _1,\gamma _2)\in {\mathcal G}^{(2)}}{\gamma _1\gamma _2=\gamma }f(\gamma _1)g(\gamma _2).
  $$
  and the involution $f^*(\gamma ) = f(\gamma^{-1})^*$.
When $\G$ is Hausdorff, $A_K(\mathcal G)$ is just the $*$-algebra of
compactly supported, locally constant functions $f\colon \mathcal G
\to K$.

It is interesting to notice that $1_B 1_D = 1_{BD}$, whenever $B$
and $D$ are compact open bisections in ${\mathcal G}$.}
\end {definition}



\section{The inverse semigroup associated to an adaptable separated graph}
\label{sect:def-inverse-semigroupS} In this section we define the semigroup $S(E,C)$ associated to
an adaptable separated graph $(E,C)$, and show that it is an $E^*$-unitary inverse semigroup.

Recall that an {\it inverse semigroup} is a semigroup $S$ such that for each $s\in S$ there is a unique $s^*\in S$ such that
$s=ss^*s$ and $s^*=s^*ss^*$, see \cite{Lawson}. Note that any two idempotents of an inverse semigroup commute (\cite[Theorem 3]{Lawson}), and the set $E(S)$ of all idempotents of $S$ is a meet-semilattice
with the operation given by the product of $S$.

A $*$-semigroup is a semigroup $S$ endowed with a unary operation $*$ such that $s^{**}= s$ and $(st)^*= t^*s^*$ for $s,t\in S$.
We first define $S(E,C)$ as an abstract $*$-semigroup by using generators and relations, and then
we prove that it is isomorphic to a concrete $*$-semigroup defined in terms of paths and monomials
in our separated graph. This will be very useful in showing that $S(E,C)$ is an inverse semigroup
and characterizing the semilattice of idempotents of $S(E,C)$, which is a key step for all what
follows.

\subsection{Definition of the inverse semigroup \texorpdfstring{$S(E,C)$}{S(E,C)}}
We start by defining $S(E,C)$ by generators and relations. Roughly, it is generated by elements
$v$ indexed by vertices of $E$, elements $e$ and $e^*$ indexed by edges of $E$, and some
additional variables $t^v_i$ and $(t^v_i)^{-1}$ whose role is to ``tame'' the relations associated
to the separated graph. Specifically, for each $v\in E^0$, we consider a collection of mutually
commuting elements $\{ t^v_i, (t^v_i)^{-1}\mid i\in \N \}$ such that
$t_i^v(t_i^v)^{-1}=v=(t_i^v)^{-1}t_i^v$, and $(t^v_i)^*=(t_i^v)^{-1}$. The way these elements
interact with the natural generators corresponding to the separated graph is specified below.
Before we give the definition of $S(E,C)$, let us fix some notation.

 \begin{notation}

            {\rm Assume the notation provided in Definition \ref{def:adaptable-sepgraphs}. If $p\in I$ is {\bf non-minimal} and {\bf free}, we denote by  $\sigma^p$ the map $\N \to \N$ given by
            $$\sigma^p (i) = i+k(p)-1.$$
            Moreover, if $1\le j \le  k(p)$, we denote by $\sigma_j^p$ the unique bijective, non-decreasing map from $\{1, \dots ,k(p)\}\setminus \{j\}$ onto $\{1, \dots , k(p)-1\}$.}

 \end{notation}

 We are now ready for the definition of the semigroup $S(E,C)$. Recall from Definition~\ref{def:adaptable-sepgraphs} the definition of an adaptable separated graph.

 \begin{definition}
 {\rm  Given an adaptable separated graph $(E,C)$, denote by $S(E,C)$ the $*$-semigroup (with $0$) generated by $E^0\cup E^1\cup \{(t_i^v)^{\pm 1} \mid i\in \N, v\in E^0 \}$ and with defining relations given by all relations
 in~{\bf\ref{pt:KeyDefs}}
 except {\bf\ref{pt:KeyDefs}}(ii)(d)~and~{\bf\ref{pt:KeyDefs}}(1)(ii).}
 \end{definition}

\begin{point}[Relations] \label{pt:KeyDefs} {\rm There are two blocks of relations in which we are interested. In the first block we write the natural relations arising from the separated
graph structure (cf. \cite{AG12}). In the second block, we specify the relation between the generators of $S(E,C)$, using the special form of our adaptable separated graph.

{\bf Block 1}:

\begin{enumerate}[\rm (i)]
\item For all $v, w\in E^0$, we have $v\cdot w=\delta_{v,w}v$ and $v=v^*$.
\item For all $e\in E^1$, we have:
\begin{enumerate}[\rm (a)]
\item $e=s(e)e=er(e)$
\item $e^*e=r(e)$
\item $e^*f=\delta_{e,f}r(e)$ if $e,f\in X\subseteq C_{s(e)}$.
\item $v=\sum_{e\in X}ee^*$, for $X\in C_v$, $v\in E^0$.
\end{enumerate}
\end{enumerate}

{\bf Block 2}:

\begin{enumerate}
 \item For each {\bf free} prime $p\in I$ and $i=1,\ldots,k(p)$,  we have:
  \begin{enumerate}[\rm (i)]
   \item $ \quad \alpha(p,i)^*\alpha(p,i)=v^p$
    \item $$\alpha(p,i)\alpha(p,i)^*=v^p-\sum^{g(p,i)}_{t=1} \beta(p,i,t)\beta(p,i,t)^* $$
     \item For $i\neq j$, $\,\, \alpha(p,i)\alpha(p,j)=\alpha(p,j)\alpha(p,i)$,   and   $\,\, \alpha(p,i)\alpha(p,j)^*=\alpha(p,j)^*\alpha(p,i)$.
  \item $\beta(p,i,s)^*\beta(p,j,t)=0$ if either $i\neq j$,  or $i=j$ and $s\neq t$. (Note that when $i=j$ and $s\neq t$, these relations follow from the separated graph relations).
  \item $\alpha (p,i)^* \beta (p,i,t) = 0 = \beta (p,i,t)^* \alpha (p,i) $ for all $1\le  i \le k(p)$ and all $1\le t\le g(p,i)$.
  
\vspace{0.2cm}
  Note that relations (i), (ii) and (v) follow from the separated graph relations, i.e., from the relations given in Block 1.
 \end{enumerate}
\item Moreover, in terms of the $\{t^v_{i}\}$, we impose the following relations:
\begin{enumerate}[\rm (i)]
\item For each $v\in E^0$, $\{ (t_i^v)^{ \pm 1} :  i\in \N \}$ is a family of mutually commuting elements such that
$$ vt_i^v = t_i^v =t_i^v v,\qquad t_i^v(t_i^v)^{-1} = v= (t_i^v)^{-1}t_i^v,\qquad (t_i^v)^*= (t_i^v)^{-1} .$$
\item If $p\in I$ is {\bf regular}, $e\in E^1$ is such that $s(e) \in E_p^0$ and $i\in \N$,
$$t^{s(e)}_i e = e t^{r(e)}_i.$$
\item If $p\in I$ is {\bf free}, $i\in \N$,  $1\le j \le k(p)$ and  $1\le s \le g(p,j)$,
$$(t_i^{v^p})^{\pm 1} \beta (p,j,s) = \beta (p,j,s) (t^{r(\beta (p,j,s))}_{\sigma^p (i)})^{\pm 1}, $$
\item If $p\in I$ is {\bf free}, $i\neq j$, and $1\le s \le g(p,j)$,
   $$\alpha(p,i)\beta(p,j,s)=\beta(p,j,s)t^{r(\beta(p,j,s))}_{\sigma^p_j(i)},  \text{ and }\,\,  \alpha (p,i)^* \beta (p,j,s) = \beta(p,j,s)(t^{r(\beta(p,j,s))}_{\sigma^p_j(i)})^{-1}.$$
   \item If $p\in I$ is {\bf free}, $t^{v^p}_i\alpha(p,j)=\alpha(p,j)t^{v^p}_i$ and $t^{v^p}_i\alpha(p,j)^*=\alpha(p,j)^*t^{v^p}_i$ for all $i\in\mathbb N$ and $j\in \{ 1,\dots , k(p) \}$.
\end{enumerate}
\end{enumerate}}
 \end{point}

 \begin{remark}
  \label{rem:starrelationsOK} {\rm Since we are working within the category of $*$-semigroups, the  $*$-relations of all the relations described in~{\bf\ref{pt:KeyDefs}} are indeed enforced in
  the $*$-semigroup $S(E,C)$.}
 \end{remark}

 We next plan to provide a different description of $S(E,C)$. This will be given via the paths that one can intuitively associate to any
 adaptable separated graph. We show in Proposition~\ref{prop:iso-with-universal} that $S(E,C)$ is isomorphic to this semigroup, which we will momentarily denote by $S$.

We must first introduce the notion of a \emph{connector path}, which we will abbreviate to c-path,
and also a standard form in which c-paths can be written. Informally, a c-path is any path
$\gamma$ in $E$ with the following two properties: firstly, if $\gamma$ is not a vertex, then its
final edge is a connector; and secondly, whenever $\gamma$ visits the vertex $v^p$ corresponding
to a free prime $p \in \Ifree$, there is a single $i \le k(p)$ such that the segment of $\gamma$
based at $v^p$ has the form $\alpha(p,i)^n$ and is followed by one of the corresponding connectors
$\{\beta(p, i, j) : j \le g(p,i)\}$. Since every nontrivial c-path ends in a connector, it can be
factorised into a standard form $\gamma = \widehat{\gamma}_1 \cdots \widehat{\gamma}_j$ so that the final
edge of each $\widehat{\gamma}_j$ is a connector, and the initial segment of $\widehat{\gamma}_j$ before
its final edge is a finite path within a single transitive component of $E$.
%
%
%
More formally, we make the following definition.

\begin{definition}[c-paths]
\label{def:finitepath}

{\rm Let $(E,C)$ be an adaptable separated graph. Then, we define a {\bf step} from a vertex $v\in E_p^0$ to a vertex $w\in E_q^0$ with $q<p$, denoted by $\widehat\gamma_{v,w}$,as follows:
 \begin{enumerate}
  \item if $v=v^p$  for $p$ a free prime, then a step from $v^p$ to $w$ is defined as $$\widehat\gamma_{v,w}:=\alpha(p,i)^m\beta(p,i,t)\text{ for some }i \text{ and some } m\ge 0, \text{ where } r(\beta(p,i,t))=w.$$
  \item if $v\in E^0_p$ for a regular prime $p$, then a step from $v$ to $w$ is defined as $$\widehat\gamma_{v,w}:=\mathbf \gamma\beta , \text{ with }s(\beta) =v', r(\beta)=w ,$$
  where $\gamma$ is a directed path of finite length connecting $v$ and $v'$ in $E_p$, and $\beta$ is a connector from $v'$ to $w$.
 \end{enumerate}
Then, given two vertices $v\in E_p^0$ and $w\in E_q^0$ with $p>q$ in $I$, a \emph{c-path} from $v$
to $w$ is a concatenation of steps starting at $p$ and ending at $q$; that is, a path of the form
\begin{equation}\label{eq:standard form}
    \gamma_{v,w}:=\widehat\gamma_{v_0,v_1}\ldots\widehat\gamma_{v_{n-1},v_n}
\end{equation}
such that there exist a sequence $p=q_0 > q_1>q_2>\ldots>q_n=q$ in $I$ and vertices $v_i\in
E^0_{q_i}$, with $v_0=v$ and $v_n=w$, such that each $\widehat{\gamma}_{v_i, v_{i+1}}$ is a step from
$v_i$ to $v_{i+1}$. Given a c-path $\gamma$ with source $v$ and range $w$, the
factorisation~\eqref{eq:standard form} is unique, and is called the \emph{standard form} of
$\gamma$.

The \emph{depth} of the $c$-path $\gamma_{v,w}$ is the number of steps in the c-path; so if
$\gamma_{v,w} = \widehat\gamma_{v_0,v_1}\ldots\widehat\gamma_{v_{n-1},v_n}$ in standard form, then the
depth of $\gamma_{v,w}$ is $n$.

A {\bf trivial} c-path is a single vertex $v\in E^0$, and these trivial c-paths have depth 0.

Given two c-paths $\gamma_1,\gamma_2$, we will write $\gamma_1\prec\gamma_2$ if $\gamma_2=\gamma_1
\eta$ for some other c-path $\eta$; we write $\gamma_1\not\prec\gamma_2$ otherwise.}
\end{definition}

\begin{remark}
{\rm Our c-paths are only some of the allowed paths of finite length in $E$. 
The usual length of a path of finite length
$\gamma $  will be denoted by $|\gamma |$. Later on, we will also give a technical sense to the
following terms: infinite path, semifinite path, and $\mathcal E$-path.}
\end{remark}

\begin{definition}[Monomials]
\label{def:monomial}{\rm
We continue with our standing assumptions on $(E,C)$.
We now define the monomials as the possible multiplicative expressions one can form using generators (excluding connectors) corresponding to a given prime.
They will be denoted by $\mathbf m(p)$ for $p\in I$. Namely,
\begin{enumerate}
 \item if $p$ is a {\bf free} prime, we define $$\mathbf m(p)=(t^{v^p}_{i_1})^{d_1}\ldots (t^{v^p}_{i_r})^{d_r}\prod^{k(p)}_{j=1}\alpha(p,j)^{k_j}(\alpha(p,j)^*)^{l_j}, \, \, d_1,\ldots,d_r\in\mathbb Z\setminus\{0\},r\geq 0,
 k_j,l_j\ge 0$$
 \item if $p$ is a {\bf regular} prime, we define $$\mathbf m(p)=(t^{v}_{i_1})^{d_1}\ldots (t^{v}_{i_r})^{d_r}\gamma\nu^*,$$ where $\gamma,\nu$ are paths of finite length in $E_p$ satisfying $s(\gamma)=v$, $v\in E^0_p$,
 and $r(\gamma)=r(\nu)$.
\end{enumerate}
For $p$ a free prime, two monomials $\mathbf m (p)$ and $\mathbf m'(p)$ corresponding to $p$ are equal if they have the same exponents of $t^{v^p}_i, \alpha (p,j)$ and $\alpha (p,j)^*$.
For $p$ a regular prime, two monomials $\mathbf m (p)$ and $\mathbf m'(p)$ corresponding to $p$ are equal if they have the same exponents of $t^{v}_i$, and $\gamma = \gamma '$, $\nu = \nu'$, where $\gamma ,\nu$
correspond to $\mathbf m (p)$ and $\gamma', \nu'$ correspond to $\mathbf m'(p)$. }
\end{definition}

Guided by the multiplication rules defined in~{\bf\ref{pt:KeyDefs}}, we will define an inverse
semigroup $S$ using c-paths and monomials. For the moment, let us deal with the set of monomials
at a given $p\in I$.

\begin{lemma}
 \label{lem:mult-of-monomials}
 For each $p\in I$, the set of monomials at $p$, together with $\{ 0 \}$ in case $p\in \Ireg$, forms a $*$-semigroup. Moreover we have
 \begin{equation}
  \label{eq:reg-for-monomials}
\mathbf m (p) \mathbf m (p)^* \mathbf m (p) = \mathbf m (p)
    \end{equation}
  for each monomial $\mathbf m (p)$.
 \end{lemma}

\begin{proof}
 Assume first that $p$ is regular. Then, by using our assumption that $|s_{E_p} (v)| \ge 2$ for all $v\in E_p^0$ (Definition~\ref{def:adaptable-sepgraphs}(2)),
 we see, using \cite[Corollary 1.5.12]{AAS}, that the set of monomials $\mathbf m (p)$ at $p$, together with $\{ 0 \}$, can be identified with a $*$-subsemigroup of the multiplicative semigroup of the
 Leavitt path algebra $L_{\Z[t_i^{\pm 1}]}(E_p)$, where $\Z [t_i^{\pm 1}]$ is the ring of
 Laurent polynomials on $\{ t_i \mid i\in \N \}$. Under this correspondence $(t_i^v)^{\pm 1} \longleftrightarrow (t_i)^{\pm 1} \cdot v$ for each $v\in E^0_p$.
 Equality~\eqref{eq:reg-for-monomials} is easily verified.

 Suppose now that $p$ is free. Then, the multiplication of $p$-monomials is defined using the rules in~{\bf\ref{pt:KeyDefs}}. Concretely, suppose that
 $$\mathbf m_s (p) = \prod _i (t^{v^p}_{i})^{d_i^{(s)}} \prod^{k(p)}_{j=1}\alpha(p,j)^{k_j^{(s)}}(\alpha(p,j)^*)^{l_j^{(s)}} $$ for $s=1,2,3$, then
 $$\mathbf m_1 (p)\mathbf m_2 (p) = \mathbf m_3 (p)$$
 if and only if $d_i^{(3)}= d_i^{(1)} + d_i^{(2)}$ for all $i\in \N$ and $k_j^{(3)}= \text{max} \{ k_j^{(1)}, k_j^{(1)}+k_j^{(2)}-l_j^{(1)} \}$, and
 $l_j^{(3)}= \text{max} \{ l_j^{(2)}, l_j^{(2)}+l_j^{(1)}-k_j^{(2)} \} $ for each $j\in \{ 1,\dots , k(p) \}$.
 Associativity, as well as the formula~\eqref{eq:reg-for-monomials}, are easily checked.
  \end{proof}

Using this whole data, we define the desired semigroup $S$. First, we fix the set:

\begin{definition}\label{Def:InverseSemigroup-Set}
{\rm Let $(E,C)$ be an adaptable separated graph. Then, we define $S$ to be the union of $\{ 0 \}$
and the set of all triples $(\gamma , \mathbf m (p) ,\eta)$, where $\gamma, \eta $ are 
c-paths (Definition~\ref{def:finitepath}), $\mathbf m (p)$ is a monomial at some prime $p\in I$
(Definition~\ref{def:monomial}), and $r(\gamma ) = s(\mathbf m (p))$, $r(\eta ) = r(\mathbf m
(p))$. So, $S$ consists of combinations of c-paths and monomials $\mathbf m$. We shall use a less
formal notation, writing the elements of $S$ just as concatenations $\gamma \mathbf m (p) \eta^*$
of a c-path, a monomial and the star of a c-path. Note that c-paths do not involve the variables
$t_i^v$, whilst monomials may involve these variables.}
\end{definition}


Our next goal is to show that $S$ is a $*$-semigroup, obeying the multiplication rules stated
in~{\bf\ref{pt:KeyDefs}}. It is important to highlight here that {\it only the purely
multiplicative relations} in~{\bf\ref{pt:KeyDefs}} are used to define the product of $S$.
Explicitly, this means that relations
{\bf\ref{pt:KeyDefs}}(ii)(d)~and~{\bf\ref{pt:KeyDefs}}(1)(ii) are not used to define it. In order
to define its product, let us firstly define what we call {\bf the translation part}
$\phi_\eta(\mathbf m(p))$ for $\eta$ a c-path from $v\in E^0_{p}$ to $w\in E^0_{p'}$, with $p>p'$,
and $\mathbf m(p)$ a monomial at $p$. This translation part is a monomial involving only the
variables $(t_i^{r(\eta)})^{\pm 1}$, designed to satisfy the relation
\begin{equation}
\label{eq:commuting-rels}
 \mathbf m (p) \eta = \tilde{\eta} \phi_\eta(\mathbf m(p)).
\end{equation}
for a suitable c-path $\tilde{\eta}$, in case the product $\mathbf m (p) \eta$ is nonzero.


Assume first that $p$ is a free prime.  By definition, the  monomial $\mathbf m (p)$ is of the form
$$\mathbf m(p)=(t^{v^{p}}_{i_1})^{d_1}\ldots (t^{v^{p}}_{i_r})^{d_r}\prod^{k(p)}_{j=1}\alpha(p,j)^{k_j}(\alpha(p,j)^*)^{l_j}$$
and the c-path to $w\in E^0_{p'}$ has standard form
$$\eta=\widehat\gamma_{v^{p},v_1}\widehat\gamma_{v_1,v_2}\ldots\widehat\gamma_{v_{n},w}=\alpha(p,i)^m\beta(p,i,t)\widehat\gamma_{v_1,v_2}\ldots\widehat\gamma_{v_{n},w}\text{
with }r(\beta(p,i,t))=v_1.$$ We write $\eta=\eta_1\eta_2$, where
$\eta_1=\alpha(p,i)^m\beta(p,i,t)$ and $\eta_2=\widehat\gamma_{v_1,v_2}\ldots\widehat\gamma_{v_{n},w}$.

Thus, applying the commutation rules~{\bf\ref{pt:KeyDefs}}(1)(iii), we get
\begin{align*}
\mathbf m(p)\eta & =(t^{v^{p}}_{i_1})^{d_1}\ldots (t^{v^{p}}_{i_r})^{d_r}\prod_{j=1}^{k(p)} \alpha(p,j)^{k_j}(\alpha(p,j)^*)^{l_j}\alpha(p,i)^m\beta(p,i,t)\eta_2 \\
& = (t^{v^{p}}_{i_1})^{d_1}\ldots (t^{v^{p}}_{i_r})^{d_r}\Big( \prod_{j\ne i} \alpha(p,j)^{k_j}(\alpha(p,j)^*)^{l_j}\Big) \underbrace{\alpha(p,i)^{k_i}(\alpha(p,i)^*)^{l_i}\alpha(p,i)^{m}\beta(p,i,t)}_{(*)}\eta_2.
\end{align*}

Looking carefully at $(*)$, we just have two possibilities:
\begin{enumerate}
\item if $l_i>m$, then $(*)=0$ by~{\bf\ref{pt:KeyDefs}}(1)(i),(v); and
\item if $l_i\leq m$, then $(*)=\alpha(p,i)^{k_i+m-l_i}\beta(p,i,t)$.
\end{enumerate}
Hence, in the latter case (the other is zero), one has that $$\mathbf m(p)\eta=\alpha(p,i)^{k_i+m-l_i}\beta(p,i,t)\eta_2\mathbf \phi_\eta(\mathbf m(p)),$$
where $\phi_\eta(\mathbf m(p))$ is a monomial in $(t_i^w)^{\pm 1}$ that comes from:
\begin{itemize}
\item passing $(t^{v^{p}}_{i_1})^{d_1}\ldots (t^{v^{p}}_{i_r})^{d_r}$ to the right through
    $\eta$ using repeatedly the rules in~{\bf\ref{pt:KeyDefs}}(2), and
\item passing the product $\prod _{j\ne i}\alpha(p,j)^{k_j}(\alpha(p,j)^*)^{l_j}$ to the right
    through $\eta$  using again repeatedly the rules in~{\bf\ref{pt:KeyDefs}}(2). (Note
    that~{\bf\ref{pt:KeyDefs}}(2)(iv) is used here in an essential way.)
\end{itemize}

We have thus shown that the relations~{\bf\ref{pt:KeyDefs}} force us to {\it define} the product
$\mathbf m (p) \eta $ in the following way:

\begin{definition}
 \label{lem:def-phi-free case}
{\rm  Assume that $p\in I$ is free, and adopt the above notation for $\mathbf m (p)$ and $\eta $.  Then we define the product $\mathbf m (p) \eta $ to be nonzero if and only if
 $v^{p} = s (\eta )$ and $l_i \leq m$. In this case we set
$$\mathbf m (p) \eta = \tilde{\eta} \phi_\eta(\mathbf m(p)).$$
 where $\tilde{\eta} = \alpha (p,i)^{k_i+m-l_i} \beta (p,i,t) \eta_2$ and $\phi_\eta(\mathbf m(p))$ is a monomial involving only the variables $ (t^{r(\eta)}_s)^{\pm 1}$, $s\in \N$, as described above.}
\end{definition}

We now consider the case where $p$ is a regular prime.

\begin{definition}
 \label{lem:def-phi-reg-case}
{\rm  Assume that $p\in I$ is regular, and let
 $$  \mathbf m(p)=(t^{v}_{i_1})^{d_1}\ldots (t^{v}_{i_r})^{d_r}\gamma\nu^*$$
 be a monomial at $p$, with $s(\gamma)=v$, $s(\nu)= v'$ and $v,v'\in E^0_p$. Let $\eta$ be a c-path from $v'$ to $w$ with standard form
$$\eta=\widehat\gamma_{v',v_1}\widehat\gamma_{v_1,v_2}\ldots\widehat\gamma_{v_{n},w}=\gamma'\beta \widehat\gamma_{v_1,v_2}\ldots\widehat\gamma_{v_{n},w}, $$
where $\gamma'$ is a path in the graph $E_p$ connecting $v'$ and $v''$, and $\beta$ is a connector with $s(\beta) =v''\text{ and }r(\beta)=v_1$.

The product $\mathbf m(p)\eta$ is nonzero if and only if $\gamma ' = \nu \gamma ''$, and  then
$$\mathbf m(p)\eta=(t^{v}_{i_1})^{d_1}\ldots (t^{v}_{i_r})^{d_r}(\gamma \gamma'')\beta\widehat\gamma_{v_1,v_2}\ldots\widehat\gamma_{v_{n},w}
= \tilde\eta \phi_\eta(\mathbf m(p)),$$ where  $\tilde\eta = \gamma
\gamma''\beta\widehat\gamma_{v_1,v_2}\ldots\widehat\gamma_{v_{n},w}$ and $\phi_\eta(\mathbf m(p))$ is a
monomial in $\{(t_i^w)^{\pm 1}\}$ that comes from passing $(t^{v''}_{i_1})^{d_1}\ldots
(t^{v''}_{i_r})^{d_r}$ to the right through $\beta\widehat\gamma_{v_1,v_2}\ldots\widehat\gamma_{v_{n},w}$
by repeated use of the rules in~{\bf\ref{pt:KeyDefs}}(2).}
\end{definition}

We have thus established the formula~\eqref{eq:commuting-rels} for any prime $p$ in $I$, whenever
the product $\mathbf m (p) \eta $ is nonzero. Note that $\phi _{\eta} (\mathbf m (p))$ only
depends on $\mathbf m (p)$ and on the connectors appearing in $\eta $. In particular, one has
$\phi_{\eta} (\mathbf m (p))= \phi_{\tilde\eta} (\mathbf m (p))$. Note also that
$\phi_{\eta}(\mathbf m (p)^*) = \phi_{\eta}(\mathbf m (p))^*$, so that $\phi_{\eta}$ preserves
adjoints.

Using the above partial definitions, we are now ready to fully define the multiplication of the elements in $S$.

\begin{definition}\label{def:multiplication1}
{\rm Let $\gamma_1\mathbf m(p)\eta_1^*$ and $\gamma_2\mathbf n(p')\eta_2^*$ be two elements in
$S$. We define
\[
\gamma_1\mathbf m(p)\eta_1^*\_\gamma_2\mathbf n (p')\eta_2^*
    =\begin{cases}
             \gamma_1[\mathbf m(p)\phi_{\eta_1'} (\mathbf n(p'))](\eta_2\tilde\eta_1')^* & \text{if $\eta_1=\gamma_2\eta_1'$  and $\mathbf n (p')^* \eta_1'\ne 0$}\strut  \\
             \gamma_1\tilde\gamma_2'[\phi_{\gamma_2'}(\mathbf m(p))\mathbf n(p')]\eta_2^* & \text{if $\gamma_2=\eta_1\gamma_2'$ and $\mathbf m (p) \gamma_2' \ne 0$}\strut  \\
             0 & \text{ otherwise}
    \end{cases}
\]}
\end{definition}

\begin{lemma}
Let $(E,C)$ be an adaptable separated graph. Then the set $S$, endowed with the above operation,  is a $*$-semigroup, and we have $s= ss^*s$ for all $s\in S$.
\end{lemma}

\begin{proof} The only delicate point is the associativity. Assume that we have three elements $s_i = \gamma_i \mathbf m (p_i) \eta_i^*$. Write $A= (s_1s_2) s_3$ and $B= s_1(s_2s_3)$.
We will show that if $A\ne 0$ then $B\ne 0$ and $A=B$. The result then follows from symmetry.

So assume that $A\ne 0$. We will suppose that $p_1,p_2,p_3$ are free primes, leaving the easier
case where one of them is regular to the reader. There are four cases to consider, all similar. So
we will only provide the details in one case. Recall that associativity of the semigroup of
monomials at a given prime has been established in Lemma~\ref{lem:mult-of-monomials}.

Write  $$\mathbf m (p_s) = \prod _i (t^{v^{p_s}}_{i})^{d_i^{(s)}}
\prod^{k(p_s)}_{j=1}\alpha(p_s,j)^{k_j^{(s)}}(\alpha(p_s,j)^*)^{l_j^{(s)}} $$ for $s=1,2,3$. We
suppose that $\eta _1 = \gamma_2 \eta _1'$ and that $k_{i_0}^{(2)} \le m_1$, where
\[
\eta_1' = \alpha (p_2,i_0)^{m_1} \beta (p_2,i_0,t_0)\cdots.
\]
By Definition~\ref{def:multiplication1}, we then have that $s_1s_2$ is nonzero and
$$s_1s_2 = \gamma_1 [ \mathbf m (p_1) \phi_{\eta_1'} (\mathbf m (p_2))] (\eta_2\tilde\eta_1')^*,$$
where $\tilde\eta_1'= \alpha (p_2,i_0)^{m_1+l_{i_0}^{(2)} - k_{i_0}^{(2)}}\beta (p_2,i_0,t_0)
\cdots $ (see Definition~\ref{lem:def-phi-free case}). Now assume that $\gamma _3 = (\eta_2
\tilde\eta_1')\gamma_3'$, and that $l_{i_1}^{(1)} \leq m_2 $, where  $\gamma_3' = \alpha
(p_1,i_1)^{m_2}\beta (p_1,i_1, t_1) \cdots $. Then the product $A= (s_1s_2)s_3$ is nonzero and
$$A= (\gamma_1\tilde\gamma_3')\Big[ \phi _{\gamma_3'} \Big( \mathbf m (p_1) \phi _{\eta_1'}(\mathbf m (p_2))\Big) \mathbf m (p_3) \Big] \eta_3^*.$$
Let us check that $s_2s_3 \ne 0$. First note that $\gamma _3 = \eta_2 \gamma_3''$, where $\gamma_3''= \tilde\eta_1' \gamma_3 '$.
Therefore, we have
$$\gamma_3 ''= \tilde\eta _1' \gamma_3'= \alpha (p_2,i_0)^{m_1+l_{i_0}^{(2)} - k_{i_0}^{(2)}}\beta (p_2,i_0,t_0) \cdots .$$
It follows that $s_2s_3 \ne 0$ if and only if $l_{i_0}^{(2)} \le m_1+l_{i_0}^{(2)} -k_{i_0}^{(2)}$, which is equivalent to our assumption $m_1\ge k_{i_0}^{(2)}$.
Hence $s_2s_3 $ is nonzero, and
$$s_2s_3 = (\gamma_2 \tilde\gamma_3'')[\phi _{\gamma_3''} (\mathbf m (p_2)) \mathbf m (p_3)] \eta_3^* ,$$
where $\tilde\gamma _3'' = \alpha (p_2,i_0)^{m_1+l_{i_0}^{(2)} - k_{i_0}^{(2)}+k_{i_0}^{(2)} -l_{i_0}^{(2)}}\beta (p_2,i_0,t_0)\cdots = \eta_1' \gamma_3'$. Now we have
$$\gamma_2 \tilde\gamma_3'' = (\gamma _2 \eta_1') \gamma_3' = \eta_1 \gamma_3 ',$$
and so the product $s_1(s_2s_3)$ is nonzero if and only if $m_2\ge l_{i_1}^{(1)}$, which holds by our hypothesis. So we get
$$B= s_1(s_2 s_3) = (\gamma _1\tilde\gamma_3') [ \phi_{\gamma_3'} (\mathbf m (p_1)) \phi_{ \gamma_3''}(\mathbf m (p_2)) \mathbf m (p_3)] \eta_3^* . $$
Hence,  the equality $A=B$ follows from the computation
$$ \phi _{\gamma_3'}\Big( \mathbf m (p_1) \phi_{\eta_1'} (\mathbf m (p_2))\Big) =  \phi_{\gamma_3'} (\mathbf m (p_1)) \phi_{\gamma_3'} \phi_{\tilde\eta_1'}(\mathbf m (p_2)) =
\phi_{\gamma_3'} (\mathbf m (p_1)) \phi_{\gamma_3''} (\mathbf m (p_2)).$$

The identity $s= ss^*s$ for $s\in S$ follows immediately from Lemma \ref{lem:mult-of-monomials} and the general form $s=\gamma \mathbf m (p) \eta^*$ of the elements of $S$.

This concludes the proof of the result.
\end{proof}

We have the following natural characterization of the $*$-semigoup $S$:

\begin{proposition}
 \label{prop:iso-with-universal}Let $(E,C)$ be an adaptable separated graph. Then, there is a natural $*$-isomorphism $S(E,C) \cong S$.

 \end{proposition}

\begin{proof}
 It is easy to see that there is a $*$-homomorphism $\varphi \colon S(E,C)\to S$ sending the generators of $S(E,C)$ to their canonical images in $S$. The obvious map $\psi : S\to S(E,C)$ is easily seen to be a $*$-homomorphism,
 which is clearly the inverse of $\varphi$.
\end{proof}

\subsection{Idempotents of \texorpdfstring{$S(E,C)$}{S(E,C)}}
Another crucial point of the given characterization of $S(E,C)$ is the easy description of the
idempotent elements in $S(E,C)$.

\begin{definition}
[Set of idempotents] \label{def:idempotents} {\rm We denote the set
of idempotents in $S(E,C)$ by $\mathcal E$. By the rules of
multiplication just defined, we easily deduce that the idempotents
are the elements of the form:
$$\gamma\mathbf m (p)\gamma^*,$$
 where $\gamma$ is a c-path to $v\in E^0_{p}$ and the monomial $\mathbf m(p)$ is either  equal to the product $\prod_{j=1}^{k(p)} \alpha(p,j)^{l_j}(\alpha(p,j)^*)^{l_j}$, when $p$ is free,
 or $\mathbf{\lambda\lambda^*}$, for a path of finite length $\lambda$ in the graph $E_p$, with $s(\lambda)=v$, when $p$ is regular. Hence, an idempotent never contains the variables $t_i^{v}$.}
\end{definition}

\begin{remark}\label{rmk:Projections} {\rm
\begin{enumerate}
  \item Let $\gamma_1\mathbf m(p)\gamma^*_1$ and $\gamma_2\mathbf n (p')\gamma_2^*$ be two idempotents in $S(E,C)$. Notice that $\mathbf m(p)\mathbf n (p') = 0$ if $p\neq p'$. Moreover, if $p=p'$ then $\mathbf m(p)\mathbf n(p')=
  \mathbf n(p')\mathbf m(p)$, i.e. idempotent monomials commute.
\item Notice that if $e=\gamma\mathbf m(p)\gamma^*$ is an idempotent, then $\phi_\eta(\mathbf
    m(p))$ is trivial for any compatible c-path $\eta$, and moreover $\tilde\eta  = \eta $ (see
    Definitions \ref{lem:def-phi-free case}~and~\ref{lem:def-phi-reg-case}). In particular, when
    we multiply $e\_\gamma_1\mathbf n(p')\gamma_2^*$, with $\gamma_1=\gamma\eta$ and $\eta$ a
    non-trivial c-path, we have the following:
\begin{enumerate}
\item If $p$ is {\bf free} and $\mathbf m (p) \eta \ne 0$, then
$$e\cdot \gamma_1\mathbf n(p')\gamma_2^*=\gamma\mathbf m(p)\gamma^*\gamma\eta\mathbf n(p')\gamma_2^* = \gamma \eta \mathbf n (p') \gamma_2^* = \gamma_1 \mathbf n (p') \gamma_2^*.$$
Hence, it follows that either $e \cdot \gamma_1\mathbf n (p')\gamma_2^*= \gamma_1\mathbf n (p')\gamma_2^*$ or $e\cdot \gamma_1\mathbf n(p')\gamma_2^*=0$.
\item If $p$ is {\bf regular}, when computing $e\_\gamma_1\mathbf n(p')\gamma_2^*$, one has to look at the corresponding paths of finite length inside $E_p$. In this situation, depending on them, the product might
be either zero or $\gamma_1\mathbf n(p')\gamma_2^*$, as before.
\end{enumerate}
\end{enumerate}}
\end{remark}

Gathering all the tools described, we finally conclude that $S$ is an inverse semigroup.

\begin{proposition}
 Let $(E,C)$ be an adaptable separated graph. Then, the semigroup $S(E,C)$ is an inverse semigroup.
\end{proposition}

\begin{proof}
 By \cite[Theorem 1.1.3]{Lawson}, it is enough to show that given any two elements $e,f\in \mathcal E$, then one has that $ef=fe$.
 To this end, we write $e,f$ in standard form, i.e. $e=\gamma_1\mathbf m(p)\gamma_1^*$ and $f=\gamma_2\mathbf n(p')\gamma_2^*$, where $\gamma_1$ is a c-path to $p$ and $\gamma_2$ is a c-path to $p'$.

 Assume that $ef\ne 0$. Then either $\gamma _1= \gamma _2$ or $\gamma _2 = \gamma _1 \eta $ for a non-trivial c-path $\eta$ and the product $\mathbf m (p)\eta $ is nonzero, or
 $\gamma _1 = \gamma_2 \eta$ for a non-trivial c-path $\eta$ and $\mathbf n (p')^*\eta = \mathbf n (p')\eta $ is nonzero. In the first case we obtain $p=p'$ and the result follows from the commutation of
 idempotent monomials.
 We shall see that in the second case we have $ef=f=fe$. By symmetry we will have $ef=e=fe$ in the third case.

 So assume that $\gamma_2 = \gamma _1 \eta$ for a non-trivial c-path $\eta$ and that $\mathbf m (p) \eta \ne 0$. Then, using the computation in Remark~\ref{rmk:Projections} we get
 $$ef = \gamma_1 \mathbf m (p) \gamma_1^* \gamma _2 \mathbf n (p') \gamma_2^* = \gamma_1 \eta  \mathbf n (p') \gamma_2 ^* = f.$$
 Similarly (or taking stars in the above), one gets $fe= f$.
  \end{proof}

Knowing that $S(E,C)$ is an inverse semigroup, it is natural to describe the natural order in
$\mathcal E$, the set of projections. This will play an important role in the sequel. The
following describes the order induced on $\mathcal E$. We omit the proof since it follows from the
same techniques appearing in later proofs.

\begin{lemma}\label{lem:projections1}
 Let $e,f\in\mathcal E$ described as $e=\gamma_1\mathbf m(p)\gamma_1^*$, $f=\gamma_2\mathbf n (p')\gamma_2^*$. Then $ef=0$ except in the following cases:
 \begin{enumerate}
  \item $\gamma _2 = \gamma _1 \eta $ for some {\bf non-trivial} c-path $\eta$ such that
      $\mathbf m (p) \eta \ne 0$. In this case, we have $f \leq e$.
  \item $\gamma_1= \gamma_2 \eta $ for some {\bf non-trivial} c-path $\eta$ such that $\mathbf n
      (p') \eta \ne 0$. In this case, we have $e\leq f$.
  \item $\gamma_1 = \gamma_2$. In this case we have $p=p'$, and
  \begin{enumerate}[\rm (a)]
  \item if $p$ is free, then $ef= \gamma_1 \mathbf m' (p) \gamma_1^*$, where, for each $1\le i \le k(p)$, the exponent of $\alpha (p,i)$ in $\mathbf m'(p)$ is the
  maximum between the exponents of $\alpha(p,i)$ in $\mathbf m(p)$ and in $\mathbf n (p)$.
  \item if $p$ is regular, and $\mathbf m (p) = \lambda \lambda^*$, $\mathbf n (p)= \mu \mu ^*$, then either $\mu = \lambda \nu'$, and then $f \le e$, or $\lambda = \mu \lambda'$, and then $e\le f$.
   \end{enumerate}
 \end{enumerate}
 \end{lemma}

 \begin{remark}
\label{rem:lub-in-E}
{\rm  If $p$ is a free prime, and $e= \gamma \mathbf m(p) \gamma^*$, $f= \gamma \mathbf n (p) \gamma^*$, we can also define the join $e\vee f$ of $e$ and $f$ by the formula
 $$ e \vee f =\gamma\mathbf{\overline m}(p)\gamma^*,$$
 where, for each $1\le i \le k(p)$, the exponent of $\alpha (p,i)$ in $\mathbf{\overline m}(p)$ is the minimum amongst the exponents of $\alpha (p,i)$ in $\mathbf m(p)$ and in $\mathbf n (p)$.
It is clear that $e\vee f$ is the {\bf least upper bound} of $e$ and $f$ in $\mathcal E$.}
\end{remark}


\subsection{The semigroup \texorpdfstring{$S(E,C)$}{S(E,C)} is \texorpdfstring{$E^*$}{E*}-unitary}
Continuing our analysis of the semigroup $S(E,C)$, in this last part of the section we show that
it is $E^*$-unitary; recall that an inverse semigroup with zero $S$ is {\it $E^*$-unitary} if for
any $s\in S$, $e\in \mathcal{E}\setminus \{ 0 \}$, $e\leq s$ implies $s\in \mathcal{E}$. This will
have crucial implications later.

\begin{proposition}
\label{prop:E-star-unitary-general}
Let $(E,C)$ be an adaptable separated graph. Then the associated inverse semigroup $S(E,C)$ is $E^*$-unitary.
\end{proposition}
\begin{proof}
Fix a nonzero idempotent $e=\gamma_1\mathbf m(p)\gamma_1^*$ and an element $s=\gamma_2\mathbf
n(q)\gamma_3^*$, and suppose that $e\leq s$; that is, $e=es$. We will distinguish several cases.

\begin{itemize}
\item $(\gamma_1\prec\gamma_2)$. By assumption, we have $\gamma_2=\gamma_1\gamma$, for some
    non-trivial c-path $\gamma$; hence,
$$\gamma_1\mathbf m(p)\gamma_1^*=\gamma_1\mathbf m(p)\gamma_1^*\gamma_2\mathbf n (q)\gamma_3^*=\gamma_1\tilde\gamma [\phi_{\gamma}(\mathbf m(p))\mathbf n(q)]\gamma_3^*.$$
Since $e$ is an idempotent, this is equal to $\gamma_2\mathbf n(q)\gamma_3^*$. Therefore,
$\gamma_1=\gamma_2=\gamma_3$, contradicting $\gamma_1\prec\gamma_2$.

\item $(\gamma_1=\gamma_2)$. In this case, we have $$\gamma_1\mathbf
    m(p)\gamma_1^*=\gamma_1\mathbf m(p)\mathbf n(q)\gamma^*_3,$$ obtaining that
    $\gamma_1=\gamma_2=\gamma_3$, $p=q$ and $\mathbf m = \mathbf m \mathbf n$. We consider two
    cases:
\begin{enumerate}[\rm (i)]

\item If $p\in I$ is {\bf free}, then by the description of the monomials in this case, one obtains that $\mathbf n$ is an idempotent monomial.  Therefore, $s$ is an idempotent, as desired.

\item If $p\in I$ is {\bf regular}, then by the description of the monomials, it follows that
    $\mathbf m(p)=\alpha_1\alpha_1^*$ and $\mathbf n (q)=\alpha_2\alpha_3^*$ for some paths
    $\alpha_1,\alpha_2,\alpha_3$ in $E_p$. Hence,
    $$\alpha_1\alpha_1^*=\alpha_1\alpha_1^*\alpha_2\alpha_3^*,$$ which implies that
    $\alpha_3=\alpha_2$; hence, $s$ is an idempotent.

\end{enumerate}
\item $(\gamma_2\prec\gamma_1)$. In this case, let us write $\gamma_1=\gamma_2\eta$ for some
    non-trivial c-path $\eta$. We have $$\gamma_1\mathbf m(p)\gamma_1^*=\gamma_1\mathbf
    m(p)\eta^*\mathbf n(q)\gamma_3^*=\gamma_1[\mathbf m(p)\phi_{\eta}(\mathbf
    n(q))](\gamma_3\tilde\eta)^*.$$ We get $\gamma_1=\gamma_2\eta =\gamma_3\tilde\eta $. Since
    $\operatorname{depth}(\tilde \eta )= \operatorname{depth} (\eta )$, we have that
    $\operatorname{depth} (\gamma_3) = \operatorname{depth} (\gamma_2 )$, and consequently we
    obtain $\gamma_2=\gamma_3$ and $\tilde \eta = \eta $. Moreover, $\phi_{\eta}(\mathbf
    n(q))=1$. Now, we separate two cases:

\begin{enumerate}[\rm (i)]
\item If $q\in I$ is {\bf free}, then, writing $\eta = \alpha (p,i)^m \beta (p,i,s) \cdots $,
    one gets from $\eta = \tilde\eta$ that the exponents of $\alpha (p,i)$ and of $\alpha
    (p,i)^*$ in $\mathbf n (q)$ are equal. Now, since $\phi_{\eta}(\mathbf n(q))=1$, we obtain
    that, for $j\ne i$, also the exponents of $\alpha (p,j)$ and $\alpha (p,j)^*$ in $\mathbf
    n (q)$ are equal, and $\mathbf n (q)$ does not involve the variables $t^{v^q}_i$. Hence,
    $\mathbf n(q)$ is an idempotent monomial. Consequently, $s=\gamma_2\mathbf
    n(q)\gamma_2^*\in \mathcal E$.

\item If $q\in I$ is {\bf regular}, then $\mathbf n(q)=\alpha\mu^*$. Write $\eta = \alpha \mu'
    \beta \cdots$, where $\beta $ is the first connector appearing in the expression of
    $\eta$. Recalling that $\mathbf n (q)^* \eta = \tilde\eta \phi_{\eta} (\mathbf n (q)^*) =
    \eta $, we obtain
$$\mu \mu' \beta \cdots = (\mu \alpha^*) (\alpha \mu ' \beta \cdots ) = \alpha \mu' \beta \cdots ,$$
so that $\mu = \alpha$ and $\mathbf n(q) = \alpha \alpha^*$ is an idempotent monomial, showing
that $s$ is idempotent.\qedhere
\end{enumerate}
\end{itemize}
\end{proof}

\begin{remark}
{\rm The anonymous referee raised the interesting question of whether $S(E,C)$ is in fact
\emph{strongly} $E^*$-unitary in the sense of \cite{BFG, Lawson} (see also \cite{MS}), but we have
not been able to determine the answer.}
\end{remark}

As a consequence of Proposition~\ref{prop:E-star-unitary-general} and \cite[Proposition 6.4 and
6.2]{ExelBraz}, we obtain the following result about the Hausdorff property of the tight groupoid
$\mathcal G_{tight}(S(E,C))$ associated to the inverse semigroup $S(E,C)$ (see
Section~\ref{Sect:Algebra-Semigroup} for further details).

\begin{corollary}
Suppose that $(E,C)$ is an adapated separated graph. Then the tight groupoid $\mathcal
G_{tight}(S(E,C))$ associated to $S(E,C)$ is Hausdorff.
\end{corollary}


\section{Filters of the inverse semigroup \texorpdfstring{$S(E,C)$}{S(E,C)}}
\label{sect:filters}

In this section we study the filters associated to the inverse semigroup $S(E,C)$, and we characterize its tight filters. They will be useful in the next section
to determine the groupoid induced by the action of the inverse semigroup on the set of tight filters.

We first recall the definition of a filter in a semilattice $\mathcal E$ of idempotents.
\begin{definition}
\label{def:filters}
A filter on $\mathcal E$ is a non-empty subset $\eta\subset \mathcal E$ such that :
 \begin{enumerate}
  \item $0\not\in\eta$.
  \item If $e\in\eta$, and $f\in\mathcal E$ satisfy $e\leq f$, then $f\in \eta$.
  \item If $e,f\in \eta$, then $ef\in \eta$.
 \end{enumerate}
We will denote the set of filters on $\mathcal E$ as $\widehat{\mathcal E_0}$.
\end{definition}

From now on, and in accordance with the notation established in
Section~\ref{sect:def-inverse-semigroupS}, we let $\mathcal E$ denote the semilattice of
idempotents of the inverse semigroup $S(E,C)$ associated to a fixed adaptable separated graph
$(E,C)$.

Recall that we can endow the set $\widehat{\mathcal E_0}$ of filters with a natural topology such that
it becomes a totally disconnected locally compact Hausdorff space. A basis of this topology is
given as follows. For finite sets $X,Y\subseteq \mathcal E$, define
\[
\mathcal U(X,Y)=\{\eta\in\widehat{\mathcal E_0}\mid X\subseteq \eta \text{ and }Y\cap\eta=\emptyset \}.
\]
Then, $\{\mathcal U(X,Y)\mid X,Y\subseteq \mathcal E \text{ are finite}\}$ is a basis for the
abovementioned topology. Note that we may assume that $|X|=1$ by replacing $X$ by the singleton consisting of its meet.

Let us provide a concrete description of any $\eta\in\widehat{\mathcal{E}_0}$ using the concrete form
of the idempotent elements. Recall that an idempotent $e\in\mathcal E$ is of the form
$\gamma\mathbf m(p)\gamma^*$ (see Definition~\ref{def:idempotents}), and that
 $\operatorname{depth}(\gamma )$ denotes the depth of the c-path $\gamma$. Also, remembering that the poset $I$ is finite,
we see that the depth of any possible c-path $\gamma$ is bounded on $\mathbb N$.

In order to understand the filters, it is convenient to introduce the following definition.

\begin{definition}
 \label{def:semfinite-path}
{\rm  Let $\gamma $ be a c-path. A {\bf semifinite path} $\mu$ starting at $\gamma $ is one of the
following:
\begin{enumerate}
 \item If $r(\gamma )= v^p$, with $p$ a free prime, then
 $$\mu = \gamma \prod_{j=1}^{k(p)} \alpha (p,j)^{k_j},$$
 where $0\le k_j \le \infty $ for all $j\in \{1,\dots , k(p) \}$. We say that $\mu$ is an {\bf infinite path} if $k_j= \infty$ for all $j\in \{ 1,\dots , k(p) \}$.
 \item If $r(\gamma ) = v$ with $v\in E_p^0$ and $p$ a regular prime, then
 $$\mu = \gamma \lambda , $$
 where $\lambda $ is either a finite or an infinite path in the graph $E_p$. We say that $\mu$ is an {\bf infinite path} if $\lambda $ is an infinite path in $E_p$.
   \end{enumerate}}
 \end{definition}

\begin{definition}
 \label{def:initial-segment}{\rm    An {\bf initial segment} of the semifinite path $\mu$ is a semifinite path $\mu'$ of the form
 $\mu' = \gamma' \lambda '$, where $\gamma '$ is a c-path  such that $\gamma = \gamma ' \gamma ''$ for some c-path $\gamma ''$, and either:
 \begin{enumerate}
 \item $\gamma = \gamma '$, and
 \begin{enumerate}[\rm (i)]
 \item if $p$ is free, then $\lambda ' = \prod_{j=1}^{k(p)} \alpha (p,j)^{l_j}$ with $l_j\in
     \Z^+$ and $l_j\le k_j$ for all $j=1,\dots , k(p)$, and
 \item if $p$ is regular, then $\lambda'$ is an initial segment of the path $\lambda$ in the
     graph $E_p$; or
 \end{enumerate}
\item  $\gamma ' \ne \gamma$, and

\begin{enumerate}[\rm (i)] \item if $\gamma '' = \alpha (q,i)^m \beta (q,i,s) \cdots$,
 where $q\in \Ifree$, then $\lambda ' = \prod_{j=1}^{k(q)} \alpha (q,j)^{t_j}$ with $t_j\in
 \Z^+$, $j=1,\dots , k(q)$, and $t_i\le m$, and
 \item  if $[r(\gamma ')]\in \Ireg$, then $\lambda'$ is an initial segment of $\gamma''$ that
     does not contain a connector.
  \end{enumerate}
  \end{enumerate}}
 \end{definition}

With the given definitions of c-paths and semifinite paths, all c-paths are semifinite. However, there
are semifinite paths of finite length which are not c-paths according to our definition.

Our goal is to show that the filters on $\mathcal E$ correspond to the collection of all
semifinite paths. To this end, we need the following definition. Let $\mu = \gamma \lambda $ be a
semifinite path. Then, for each initial segment $\mu'= \gamma' \lambda'$ of $\mu$, we will denote
the idempotent arising from $\mu'$ as:
$$e(\mu ') := \gamma' \lambda' (\lambda')^* (\gamma')^* \in \mathcal E.$$
Using this notion, we show the desired equivalence.

\begin{theorem}
 \label{thm:filters-and-semif-paths} Let $\mathcal S$ be the collection of all semifinite paths. Then there is a bijective correspondence
 $$\varphi  \colon \mathcal S \to \widehat{\mathcal E _0}$$
  such that, for $\mu \in \mathcal S$,
  $$\varphi (\mu ) = \{ e(\mu ') \mid  \mu '  \text{ is an initial segment of } \mu \}.$$
  \end{theorem}

\begin{proof}  It is easy to see that $\varphi (\mu)$ is a filter, and that the map $\varphi$ is one-to-one.

It remains to check that $\varphi$ is surjective. Let $\eta$ be a filter on $\mathcal E$. Observe
that, since $I$ is a finite poset, there is a maximum for the depths of the c-paths appearing in
the expressions of the elements of $\eta $.  Let $e= \gamma \mathbf m (p) \gamma ^*$ be an element
in $\eta$ such that $\gamma $ has maximum depth, say $k$. We distinguish two cases. If $p$ is a
regular prime, then there exists a path (of finite or infinite length) $\lambda $ in $E_p$ such
that
$$\{ \lambda' \mid \gamma \lambda'(\lambda')^* \gamma^* \in \eta \} = \{ \lambda' \mid \lambda = \lambda' \lambda '' \};$$
that is, the set of paths $\lambda'$ of finite length in $E_p$ such that $\gamma
\lambda'(\lambda')^* \gamma^* $ belongs to $\eta$ is the set of initial segments of the path
$\lambda$. Write $\mu = \gamma \lambda$ for this path $\lambda$, and observe that $\varphi (\mu) =
\eta$. If $p$ is a free prime, then for each $1\le j \le k(p)$, we let $k_j\in \Z^+ \cup
\{\infty\}$ be the supremum of the positive integers $l_j$ such that $\eta $ contains an element
of the form $\gamma \mathbf m' (p) \gamma^*$ such that the exponent of $\alpha (p,j)$ in $\mathbf
m'(p)$ is $l_j$. Now set
$$\mu = \gamma \prod_{j=1}^{k(p)} \alpha (p,j)^{k_j},$$
which is a semifinite path.

We claim that $\varphi (\mu) = \eta$. We first check that $\varphi (\mu ) \subseteq \eta$. Write
$\mu = \gamma \lambda$ and let $\mu '= \gamma ' \lambda'$ be an initial segment of $\mu$. If
$r(\gamma ')$ is a regular prime, then it follows easily that $\gamma '
(\lambda')(\lambda')^*(\gamma ')^* \in \eta$, so that $e (\mu ') \in \eta$. Assume that $q: =
r(\gamma ' )$ is a free prime. If $\gamma ' = \gamma $, then write $\mu '= \gamma \lambda '$,
where $\lambda ' = \prod_{j=1}^{k(p)} \alpha (p, j)^{l_j} (\alpha (p,j)^*)^{l_j} $, where the
$l_j$ are non-negative integers with $l_j\le k_j$ for all $j$. By definition of $k_j$, for each
$1\le j \le k(p)$ there exists $f_j = \gamma \mathbf m _j(p)  \gamma ^* \in  \eta$ such that the
exponent of $\alpha(p,j)$ in $\mathbf m_j (p)$ is at least $l_j$. Let $f = \prod _{j=1}^{k(p)}
f_j$. Since $\eta$ is a filter, we have that $f\in \eta $, and the exponent of {\it each} $\alpha
(p,j)$ in $f$ is $\ge l_j$, for $j=1,\dots , k(p)$. Hence
$$ e(\mu')f=(\gamma \lambda '(\lambda ')^* \gamma ^*)f = f.$$
Since $\eta$ is a filter, we get that $e(\mu')=\gamma \lambda '(\lambda ')^* \gamma ^*\in \eta $, as desired. Finally, assume now that $\gamma '\ne \gamma$, and let
$$\lambda ' = \prod_{j=1}^{k(q)} \alpha (q,j)^{t_j} (\alpha (q,j)^*)^{t_j},\qquad \gamma '' =\alpha (q,i)^m \beta (q,i,s)\cdots ,$$
where $\gamma = \gamma ' \gamma ''$. Since ${t_i}\le m$, we have that $e(\mu')e= e$ (see
Lemma~\ref{lem:projections1}). Therefore $e \le e(\mu')$ and since $e\in \eta$ and $\eta $ is a
filter, we get that $e(\mu')\in \eta$. Hence we have verified that $\varphi (\mu ) \subseteq
\eta$.

Now we prove that $\eta \subseteq \varphi (\mu)$. If $f= \gamma ' \mathbf n (q) (\gamma')^*$ is
another element of $\eta$, then the condition that $ef\ne 0$, Lemma~\ref{lem:projections1}, and
the maximality of $k$, give that $\gamma = \gamma' \gamma ''$ for a c-path $\gamma ''$. If $\gamma
''$ is trivial, then $\gamma = \gamma '$, and by definition of $\mu$ we have that $\mathbf n (p) =
\lambda ' (\lambda ')^*$, where $\lambda '$ is an initial segment of $\lambda $ in the graph $E_p$
if $p$ is regular, and $\lambda ' = \prod_{j=1}^{k(p)} \alpha (p, j)^{l_j} $, where $l_j$ are
non-negative integers with $l_j\le k_j$ if $p$ is free. In either case, $\mu' = \gamma \lambda '$
is an initial segment of $\mu$ and so, that $f \in \varphi (\mu)$. Assume that $\gamma ' \ne
\gamma $. If $q$ is regular, then $f= \gamma ' \lambda'(\lambda')^*(\gamma')^*$, where $\lambda '$
is a path in $E_q$ and, since $fe\ne 0$, we get that $\lambda '$ must be an initial segment of
$\gamma ''$. If $q$ is free, write $\gamma '' = \alpha (q,i)^m \beta (q,i, s) \cdots $ and
$\mathbf n (q) = \prod_{j=1}^{k(q)} \alpha (q,j)^{t_j}(\alpha (q,j)^*)^{t_j}$. Since $fe \ne 0$ we
see from Lemma~\ref{lem:projections1} that $t_i\leq m$, and thus $\mu ':= \gamma
\prod_{j=1}^{k(q)} \alpha (q,j)^{t_j}$ is an initial segment of $\mu$. Hence
\[
f = e(\mu ')\in \varphi (\mu).\qedhere
\]
\end{proof}

 \begin{corollary}
  \label{cor:ultrafilters}
The above isomorphism $\varphi$ restricts to a bijection between the set of infinite paths and the
set $\widehat{\mathcal E}_{\infty}$ of ultrafilters.
 \end{corollary}
 \begin{proof}
This follows from Theorem~\ref{thm:filters-and-semif-paths} and the fact that the ultrafilters are
precisely the maximal filters.
\end{proof}

We can now show that the ultrafilters coincide with the tight filters. Recall that, by
\cite[Theorem 12.9]{ExelBraz}, the space $\widehat{\mathcal E}_{tight}$ of tight filters coincides
with the closure of the space of ultrafilters $\widehat{\mathcal E}_{\infty} $ in the space of
filters $\widehat{\mathcal E}_0$.

\begin{theorem}
 \label{thm:tightfilters}
 The space $\widehat{\mathcal E}_{\infty}$ of ultrafilters is closed in the space $\widehat{\mathcal E}_0$ of filters. Consequently, the space of ultrafilters coincides with the space of tight
 filters.
\end{theorem}

\begin{proof}
 It is enough to show that the complement of $\widehat{\mathcal E}_{\infty}$ in $\widehat{\mathcal E}_0$ is open. Let $\eta = \varphi (\mu)$, where $\mu= \gamma \lambda $ is a semifinite path which is not an infinite path.
 If $w:= r(\gamma )$, with $w\in E^0_p$  and $p$ a regular prime, then $\lambda $ must be a finite path in the graph $E_p$. Let $X= \{ \mu \mu^* \}$ and $Y= \{ \mu e e^* \mu^* \mid e\in s^{-1}_{E}(r_{E}(\lambda))  \}$. Since $E$ is a row-finite
 graph, it follows that $Y$ is a finite set. Now
 $$\mathcal U (X,Y) = \{ \eta \},$$
 which implies that $\eta $ is an isolated point of $\widehat{\mathcal E}_0$ in this case.

 Now suppose that $p$ is a free prime, and that $w=v^p$. Then, $\lambda = \prod_{j=1}^{k(p)} \alpha (p,j)^{k_j} $, and, by assumptions, there exists an index $i_0$ such that $k_{i_0}$ is finite.
 Assume, for convenience, that $i_0=1$. Define semifinite paths $\mu'$ and $\delta$ by
 $$\mu ' :=  \gamma \alpha (p,1)^{k_1}\qquad\text{and}\qquad \delta : = \gamma \alpha (p,1)^{k_1+1}, $$
  which define idempotents $e(\mu') = \mu '(\mu')^*, e(\delta ) = \delta \delta ^*\in \mathcal E$.
 Let
 $$Y_1= \{ e(\mu' \beta (p,1,s)) \mid 1\le s\le g(p,1) \},$$
 and define
 $$X= \{ e (\mu ') \} \qquad \text{and} \qquad  Y = \{ e(\delta )\}\cup Y_1.$$

Then $\mathcal U (X,Y)$ is a neighborhood of $\eta$ consisting entirely of semifinite paths which
are not infinite. Indeed, if $\eta ' \in \mathcal U (X,Y)$, then either $\eta '$ is of the form
$\varphi (\mu '')$, where $\mu ''$ is a semifinite path of the form
$$\gamma \prod_{j=1}^{k(p)} \alpha (p,j)^{l_j},$$
where $l_1=k_1$, which is therefore not an infinite path,  or it is of the form $\varphi  ( \rho)$, where $\rho =\gamma \alpha (p,j)^m \beta (p,j,s) \rho '$, where  $j>1$ and $ \rho '$ is a semifinite path.
Notice that in this case,
$$e (\delta ) = \gamma \alpha (p,1)^{k_1+1}(\alpha (p,1)^*)^{k_1+1}\gamma ^* \in \varphi (\rho ) = \eta '$$
because $e (\delta ) \geq g$, where $g= e(\gamma \alpha (p,j)^m \beta (p,j,s))\in \eta '$. Hence, $e (\delta ) \in \eta '$, contradicting that $\eta '\in \mathcal U (X,Y)$.

So, this second case is not possible and we obtain that $\mathcal U (X,Y)$ is a neighborhood of
$\eta$ consisting entirely of semifinite paths which are not infinite.
 \end{proof}

We close this section by providing an important characterization of the topology associated to the
space $\widehat{\mathcal E}_{\infty}$.  To do so, we use the identification of the space of
ultrafilters $\widehat{\mathcal E}_{\infty}$ with the space of infinite paths.

\begin{definition}
 \label{def:familycalP}
{\rm  We denote by $\mathcal P$ the set of semifinite paths of the form $\mu  = \gamma \lambda $,
where $\gamma $ is a c-path, and $\lambda$ is a path of finite length in the component of a
regular prime, or
 $\lambda = \prod_{j=1}^{k(p)} \alpha (p,j)^{k_j}$ for $k_j\in \Z^+$, $1\le j \le k(p)$ for a free prime $p$.

 Notice that every $e\in \mathcal E$ is of the form $e(\mu )$ for a unique $\mu  \in \mathcal P$.
 Accordingly, elements of $\mathcal P$ will be called $\mathcal E$-paths.
 For $\mu\in \mathcal P$, write
 $$\mathcal Z (\mu) = \{ \eta \in \widehat{\mathcal E}_{\infty} \mid \mu \mu^*\in \eta \}.$$ Depending on the situation, $\mathcal Z(\mu)$ might also be denoted by the
 idempotent it determines, i.e., $\mathcal Z(e(\mu))$. Notice that $\mathcal{Z}(\mu)=\mathcal{U}(\{ \mu\mu^*\}, \emptyset)\cap \widehat{\mathcal{E}}_{\infty}$.}
   \end{definition}

Given finite $X,Y \subseteq \mathcal{P}$, we write $\mathcal{U}^\infty(X, Y) := \mathcal{U}(X, Y)
\cap \widehat{\mathcal{E}}_\infty$. Then the induced topology on $\widehat{\mathcal E}_{\infty}$ is
generated by $\{\mathcal U^{\infty} (X,Y) : X,Y \subseteq \mathcal{P}\text{ are finite}\}$. The
next result provides some properties of the sets $\mathcal Z (\mu) $ just defined.

\begin{proposition}
 \label{prop:basis-of-topology}
 Let $\mathcal P$ be the set of $\mathcal E$-paths. Then
 \begin{enumerate}
  \item the family $\{ \mathcal Z (\mu)\mid \mu \in \mathcal P \}\cup \{\emptyset \}$ is closed
      under finite intersections;
  \item  the family $\{ \mathcal Z (\mu)\mid \mu \in \mathcal P \}$ is a basis for the topology
      of $\widehat{\mathcal E}_{\infty}$;
 \item for $\mu, \rho \in \mathcal P$, the set $\mathcal Z (\mu) \setminus \mathcal Z (\rho)$ is
     a finite disjoint union of sets of the form $\mathcal Z (\mu')$, for $\mu'\in \mathcal P$;
     and
 \item for each $\mu \in \mathcal P$, the set $\mathcal Z (\mu)$ is open and compact.
  \end{enumerate}
\end{proposition}

\begin{proof}
 (1) This follows from the fact that $\mathcal{Z}(\mu)=\mathcal{U}(\{ \mu\mu^*\}, \emptyset)\cap \widehat{\mathcal{E}}_{\infty}$ and the general equality
 $\mathcal U (\{ e\},\emptyset) \cap \mathcal U ( \{ f \}), \emptyset) =  \mathcal U (\{ ef \},\emptyset) $ for idempotents $e,f$ in any inverse semigroup.

 (2) This is \cite[Proposition 2.5]{EP1}.

(3)  Suppose that $\mathcal Z (\mu) \setminus \mathcal Z (\rho)$ is non-empty for a given pair $\mu ,\rho \in \mathcal P$. Since
$$\mathcal Z (\mu) \setminus \mathcal Z (\rho) = \mathcal Z (\mu ) \setminus (\mathcal Z (\mu ) \cap \mathcal Z (\rho )),$$
we may assume, using (1), that $\mathcal Z (\rho ) \subset \mathcal Z (\mu)$, and hence that $\mu$ is an initial segment of $\rho$.

 Thus, $\mathcal Z (\mu) \setminus \mathcal Z (\rho)$ is the set of infinite paths which have $\mu$ as
 initial segment but do not have $\rho$ as initial segment. Now, using that $E$ is row-finite, we see that it suffices to consider the situation where there is
 a c-path $\gamma$ ending at $v^p$ for some $p\in \Ifree$, and there are exponents $l_j\le k_j$ for all $j$, with strict inequality for at least one of the indices $j\in \{1,\dots ,
 k(p)\}$, such that
 \[
 \mu = \gamma \prod_{j=1}^{k(p)} \alpha (p,j)^{l_j}\qquad\text{and}\qquad \rho = \gamma \prod_{j=1}^{k(p)} \alpha (p,j)^{k_j}.
 \]
 We have
 $$\mathcal Z (\mu) \setminus \mathcal Z (\rho) = \bigsqcup_{\{ j : l_j < k_j \}} \bigsqcup_{l_j\le t_j<k_j} \bigsqcup_{s=1}^{g(p,j)} \mathcal Z (\gamma \alpha (p,j)^{t_j} \beta (p,j,s)),$$
which gives the desired decomposition.

 (4) This is a standard argument, which we review for the sake of completeness. The topology of $\widehat{\mathcal E}_0$ is the topology induced by the embedding into the compact set $\{ 0,1 \}^{\mathcal E}$.
 The sets of the form $\mathcal U (X,\emptyset )$, for $X\ne \emptyset$, are compact in $\widehat{\mathcal E}_0$ because they are closed subsets of the basic compact and open subsets
 $$\{ f\in \{ 0,1 \}^{\mathcal E} \mid f(x) =1 \quad \forall x\in X \}$$
 of $\{ 0,1\}^{\mathcal E }$. (Note that the assumption $X\ne \emptyset$ is crucial here). Since the space $\widehat{\mathcal E}_{\infty} $ is closed in $\widehat{\mathcal E}_0$ (Theorem~\ref{thm:tightfilters}), we conclude that
 $\mathcal U^{\infty} (X, \emptyset ) = \mathcal U (X, \emptyset ) \cap \widehat{\mathcal E}_{\infty}$ is compact in $\widehat{\mathcal E}_{\infty}$.
\end{proof}

 For the next corollary, we need a straightforward lemma, probably well-known.

 \begin{lemma}
 \label{lem:disjoint-unions}
 Let $X$ be a topological space with a basis $\mathcal B$ of the topology satisfying the following properties:
 \begin{enumerate}
  \item Each $B\in \mathcal B$ is a compact open subset of $X$.
  \item $\mathcal B$ is closed under finite intersections
  \item For each $B_1,B_2\in \mathcal B$, we have that $B_1\setminus B_2$ is a finite disjoint union of sets from $\mathcal B$.
 \end{enumerate}
Then every compact open subset of $X$ is a finite disjoint union of members of $\mathcal B$.
   \end{lemma}
 \begin{proof}
  Let $U$ be a compact open subset of $X$. By compactness we can write $U=\bigcup_{i\in I} B_i$, where $I$ is finite and $B_i\in \mathcal B$.
  Now we can refine this decomposition to a disjoint decomposition into compact and open subsets of the form
  $$(B_{i_1}\cap B_{i_2}\cap \cdots \cap B_{i_r}) \setminus \Big[ \bigcup _{j\notin \{i_1,\dots ,i_r\}} (B_{i_1}\cap B_{i_2}\cap \cdots \cap B_{i_r}\cap B_j) \Big],$$
  where $i_1,\dots ,i_r$ are different elements of $I$, and $1\le r\le |I|$. (In case $r=|I|$ this term should be interpreted as $\bigcap_{i\in I} B_i$.)
  Now use the hypothesis to express each of these sets as a finite disjoint union of members of $\mathcal B$.
   \end{proof}

\begin{corollary}
 \label{cor:compact-opens}
 The space $\widehat{\mathcal E}_{\infty}$ of ultrafilters admits a basis of compact open subsets, namely the family $\{ \mathcal Z (\mu ) \}_{\mu \in \mathcal P} $.
 Moreover, every compact open subset of $\widehat{\mathcal E}_{\infty}$ is a finite disjoint union of sets of the form $\mathcal Z (\mu)$, for $\mu \in \mathcal P$.
 \end{corollary}

\begin{proof}
By Proposition~\ref{prop:basis-of-topology}, the family $\{ \mathcal Z (\mu ) \}_{\mu \in \mathcal
P} \cup \{ \emptyset \}$ satisfies the hypothesis of Lemma~\ref{lem:disjoint-unions}. Therefore
the result follows from Lemma~\ref{lem:disjoint-unions}.
\end{proof}



\section{The \texorpdfstring{$K$}{K}-algebra \texorpdfstring{$\mathcal S_K(E,C)$}{} induced by the inverse semigroup \texorpdfstring{$S(E,C)$}{S(E,C)}}\label{Sect:Algebra-Semigroup}

Following the notation of section~\ref{sect:def-inverse-semigroupS}, given an adaptable graph
$(E,C)$, we denote its associated inverse semigroup by $S(E,C)$, and its space of tight filters by
$\widehat{\mathcal E}_{tight}$. By Corollary~\ref{cor:ultrafilters} and
Theorem~\ref{thm:tightfilters}, the set of tight filters corresponds to the set of infinite paths.
The main result of this section is Theorem~\ref{thm:Steinbergalgebra-and-our-algebra}, where we
prove that the Steinberg algebra $A_K(\mathcal G_{tight}(E,C))$ of the groupoid $\mathcal
G_{tight}(E,C)$ of germs of the canonical action of $S(E,C)$ on  $\widehat{\mathcal E}_{tight}$ is
canonically isomorphic to the $*$-algebra $\mathcal S_K(E,C)$  defined by the family of generators
$E^0\cup E^1\cup \{ (t_i^v)^{\pm 1} \}$ and the relations given in~{\bf\ref{pt:KeyDefs}}. In order to ease understanding this section we have divided it in three parts. Firstly we recall the construction of the groupoid $\mathcal
G_{tight}(E,C)$. Then, we study the behaviour of $\mathcal S_K(E,C)$, and we finish the section proving Theorem~\ref{thm:Steinbergalgebra-and-our-algebra}.
\subsection{Steinberg algebras and inverse semigroup representations}
\label{sect:Steinbergalges} In this subsection we explain the construction of the groupoid of
germs of an action of an inverse semigroup $S$ on a locally compact Hausdorff topological space $X$ (in the sense of \cite[Definition 4.3]{ExelBraz}),
and review the fact that the Steinberg
algebra is the universal *-algebra for tight representations of $S$.

Let $S$ be an inverse semigroup with $0$, and denote by $\mathcal E$ its semilattice of idempotents.
First recall that, if $F\subseteq \mathcal{E}$ is any subset, then a finite subset $\Sigma\subseteq F$ is a {\bf finite cover} of $F$ when for any $0\neq f\in F$ there exists $e\in \Sigma$ such that $fe\neq 0$.

\begin{definition}\label{def:tight} {\rm (cf. \cite{ExelBraz}, \cite[Section 5]{Steinb16}) Let $S$ be an inverse semigroup, and $A$ be a $*$-algebra over a field with involution $K$. Then, we say that $\pi:S\to A$ is a
        \emph{representation} if $\pi(st)=\pi(s)\pi(t)$ and $\pi(s^*)=\pi(s)^*$ for all $s,t\in S$. A representation $\pi$ is said to be a \emph{tight representation} if for every idempotent $e\in \mathcal E$
        and every finite cover $Z$ of $\mathcal F_e:= \{ f\in \mathcal E \mid f\le e \}$, we have
        $$\pi (e) = \bigvee_{z\in Z}\pi(z)$$
        in the (generalized) Boolean algebra of idempotents of the commutative $*$-subalgebra $A_{\mathcal E}$ of $A$ generated by $\pi (\mathcal E )$. }
\end{definition}

Given a semilattice $\mathcal E$ with $0$, the space $\widehat{\mathcal E}$ is the space of
non-zero semilattice homomorphisms $\varphi
\colon \mathcal E \to \{ 0,1 \}$, endowed with the relative topology from $\{ 0,1 \}^{\mathcal
E}$. The space of characters $\widehat{\mathcal E}_0$ is the space of those homomorphisms
$\varphi\in \widehat{\mathcal E}$  such that $\varphi (0_S)= 0$ \cite[Definition 12.4]{ExelBraz}, and may be identified
with the space of filters over $\mathcal E$ (see Definition~\ref{def:filters}) via the map that
associates to each filter $F$ its characteristic function $1_F:\mathcal E\to \{0,1\}$. A character
$\varphi$ is said to be {\it tight} if $\varphi (e) = \bigvee_{f\in Z} \varphi (f) $ for every
idempotent $e\in \mathcal E$ and every finite cover $Z$ of $\mathcal F_e$. By \cite[Proposition
11.8]{ExelBraz}, this is equivalent to the original definition of tight character in
\cite[Definition 12.8]{ExelBraz}. The space of tight characters will be identified with the space
of tight filters, and will be denoted by $\widehat{\mathcal E}_{tight}$ in accordance with the
notation employed in Section~\ref{sect:filters}.


We recall below the construction of the groupoid of germs for completeness (see \cite{ExelBraz} for further details). If $\alpha$ is an action of an inverse semigroup $S$
on a locally compact Hausdorff space $X$, then the groupoid of germs $S \times_\alpha X$ is
defined as follows: define a relation $\sim$ on $\{(s, x) \in S \times X : x \in \text{dom}(\alpha_s)\}$
by $(s, x) \sim (s', y)$ if $x = y$ and there is an idempotent $p \in E(S)$ such that $x \in \text{dom}(\alpha_p)$ and $sp = s'p$. This is an equivalence relation, and the
collection $S \times_\alpha X$ of equivalence classes for this relation is a locally
compact \'etale groupoid with unit space $X$ and structure maps $r([s,x]) =
\alpha_s(x)$, $s([s,x]) = x$, $[s,\alpha_t(x)][t, x] = [st, x]$, and $[s,x]^{-1} =
[s^*, \alpha_s(x)]$. When $X$ is the space $\widehat{\mathcal E}_{tight}$ of tight filters on the semilattice $\mathcal E$ of idempotents of $S$,
the topology of the groupoid of germs $\mathcal G_{tight}(S):= S\times_{\alpha} \widehat{\mathcal E}_{tight}$ is generated by the sets
$$ \Theta (s,U) := \{ [s,x] \in \mathcal G  \mid x\in U \},$$
where $U$ is an open subset of $\widehat{\mathcal E}_{tight}$ such that $U\subseteq D_{s^*s} := \{ \varphi \in \widehat{\mathcal E}_{tight} \mid \varphi (s^*s) = 1 \}$.
Endowed with this topology, $\mathcal G_{tight}(S)$ is an ample groupoid.
Recall that the canonical action $\alpha$ of $S$ on $\widehat{\mathcal E}_{tight}$ is given as follows. For $s\in S$, the map $\alpha_s\colon D_{s^*s}\to D_{ss^*}$ is the map defined by
$\alpha _s(\varphi )(e)= \varphi (s^*es)$ for all $\varphi\in D_{s^*s}$.

Now, given an inverse semigroup $S$ and a field with involution $K$, we may consider the Steinberg
$K$-algebra $A_K(\mathcal{G}_{tight}(S))$ associated to the ample groupoid
$\mathcal{G}_{tight}(S)$ (Definition~\ref{def:Steinbergalgebra}). If $\mathcal{G}_{tight}(S)$ is
Hausdorff, then the algebra $A_K(\mathcal{G}_{tight}(S))$ is just the algebra of $K$-valued
locally constant functions with compact support, endowed with the convolution product and the
involution $f^*(\alpha) = f(\alpha^{-1})^*$.

We say that a $*$-algebra $A$, together with a tight representation $\iota \colon S \to A$,  is {\it universal for tight representations} if given any $*$-algebra $B$ and any tight representation $\phi \colon
S\to B$, there is a unique $*$-homomorphism $\tilde{\phi}\colon A\to B$ such that $\tilde{\phi}\circ \iota = \phi$. By the usual argument, such a universal tight $*$-algebra is unique up to $*$-isomorphism.

We will need the following result, due to Steinberg \cite{Steinb16}.

\begin{theorem} \cite[Corollary 5.3]{Steinb16}
    \label{thm:steinberg16}
    Let $S$ be a Hausdorff inverse semigroup with zero and let $K$ be a field with involution. Then $A_K(\mathcal G _{tight}(S))$ is the universal $*$-algebra for tight representations of $S$.
\end{theorem}

We refer the reader to \cite{Steinb16} for the definition of a Hausdorff inverse semigroup. For us, it is enough to know that any $E^*$-unitary inverse semigroup is Hausdorff
(see \cite[page 1037]{Steinb16}).

\subsection{The \texorpdfstring{$*$}{*}-algebra \texorpdfstring{$\mathcal S_K(E,C)$}{S\_K(E,C)} and expansions}
We define the $K$-algebra $\mathcal S_K(E,C)$  as the $*$-algebra over $K$ with generators
$E^0\cup E^1\cup \{ (t_i^v)^{\pm 1} : i\in \N,  v\in E^0 \}$  and defining relations given
in~{\bf\ref{pt:KeyDefs}} (including this time {\it all} the relations); notice that $\mathcal
S_K(E,C)$ is the $K$-span of the elements of the inverse semigroup $S(E,C)$. We will show that the
natural representation of $S(E,C)$ inside $\mathcal{S}_K(E,C)$ is its universal tight
representation (Theorem~\ref{thm:tight}, Theorem~\ref{thm:Universal}).

Let us start fixing the notation  $\iota:S(E,C)\to\mathcal S_K(E,C)$ for the natural
representation of $S(E,C)$ into $\mathcal S_K(E,C)$, and recalling that all relevant notions have
been defined in Section~\ref{sect:Steinbergalges}.

\begin{remark}
\label{rk:equalities} {\rm There are a couple of equalities that
play a key role in checking tightness of $\iota$. We recall them
below for convenience.
\begin{enumerate}[\rm (a)]
\item for each $p\in I_{reg}$ and each $v\in E^0_p$, $$v=\sum_{e\in s^{-1}(v)}ee^*.$$
\item for each $p\in I_{free}$ and each $1\leq i\leq k(p)$,
    $$v^p=\alpha(p,i)\alpha(p,i)^*+\sum^{g(p,i)}_{s=1}\beta(p,i,s)\beta(p,i,s)^*$$
\end{enumerate}}
\end{remark}

Let us start by establishing the following result about suprema in $\mathcal S_K(E,C)$:

\begin{lemma}
\label{lem:preservation-suprema} For $p\in \Ifree$, let $\mathbf m_1,\mathbf m_2\in \mathcal E$
monomials associated to $p$. Then $\iota (\mathbf m_1\vee\mathbf m_2)=\iota(\mathbf m_1)\vee
\iota(\mathbf m_2)$.
\end{lemma}

\begin{proof}
We can describe $\mathbf m_1$ and $\mathbf m_2$ as follows:
$$\mathbf m_1=\prod^{k(p)}_{i=1}\alpha(p,i)^{k_i}(\alpha(p,i)^*)^{k_i}\,\,\text{ and }\,\,\mathbf m_2=\prod^{k(p)}_{i=1}\alpha(p,i)^{l_i}(\alpha(p,i)^*)^{l_i}.$$
Let $s_i:=\max\{k_i,l_i\}$ and $r_i:=\min\{k_i,l_i\}$ for $1\leq i\leq k(p)$. Then
$$\mathbf m_1\_\mathbf m_2=\prod^{k(p)}_{i=1}\alpha(p,i)^{s_i}(\alpha(p,i)^*)^{s_i}\,\,\text{ and
}\,\,\mathbf m_1\vee\mathbf m_2=\prod^{k(p)}_{i=1}\alpha(p,i)^{r_i}(\alpha(p,i)^*)^{r_i}.$$ By
using the expansion rule explained in Remark~\ref{rk:equalities} (b) as much as needed, we have:
\begin{align*}
\iota(\mathbf m_1\vee\mathbf m_2)
    = \mathbf m_1\mathbf m_2&{}+\sum_{k_i<l_i}\sum_{r=1}^{g(p,i)}\sum_{t=1}^{l_i-k_i-1}\alpha(p,i)^{k_i+t} \beta(p,i,r)\beta(p,i,r)^*(\alpha(p,i)^*)^{k_i+t}\\
                            &{}+\sum_{l_j<k_j}\sum^{g(p,j)}_{r=1}\sum_{t=1}^{k_j-l_j-1}\alpha(p,j)^{l_j+t}\beta(p,j,r)\beta(p,j,r)^{*}(\alpha(p,j)^*)^{l_j+t},
\end{align*}
where the latter computation is performed in the algebra $\mathcal S_K(E,C)$.

Now, we have:
\begin{enumerate}
\item $[\iota(\mathbf m_1)\vee\iota(\mathbf m_2)](\mathbf m_1\mathbf m_2)=\mathbf m_1\mathbf m_2.$
\item For each summand where $k_i<l_i$,
\begin{align*}
[\iota(\mathbf m_1)&{}\vee\iota(\mathbf m_2)]\left(\sum_{r=1}^{g(p,i)}\sum_{t=1}^{l_i-k_i-1}\alpha(p,i)^{k_i+t}\beta(p,i,r)\beta(p,i,r)^*(\alpha(p,i)^*)^{k_i+t}\right)\\
    &= (\mathbf m_1+\mathbf m_2-\mathbf m_1\mathbf m_2)\left(\sum_{r=1}^{g(p,i)}\sum_{t=1}^{l_i-k_i-1}\alpha(p,i)^{k_i+t}\beta(p,i,r)\beta(p,i,r)^*(\alpha(p,i)^*)^{k_i+t}\right)\\
    &={(a)}
\end{align*}
Fixing $r$ and $t$, we have
\begin{align*}
\mathbf m_1(\alpha(p,i)^{k_i+t}&\beta(p,i,r)\beta(p,i,r)^*(\alpha(p,i)^*)^{k_i+t})\\
    &= \prod^{k(p)}_{j=1}\alpha(p,j)^{k_j} (\alpha(p,j)^*)^{k_j}(\alpha(p,i)^{k_i+t}\beta(p,i,r)\beta(p,i,r)^*(\alpha(p,i)^*)^{k_i+t})\\
    &=\alpha(p,i)^{k_i+t}\beta(p,i,r)\beta(p,i,r)^*(\alpha(p,i)^*)^{k_i+t}.
\end{align*}
Indeed, for $j\neq i$, we have an absorption of the $j$-th part and, for $i=j$, one has $\alpha(p,i)^{k_i}(\alpha(p,i)^*)^{k_i}\alpha(p,i)^{k_i+t}=\alpha(p,i)^{k_i+t}.$

The same product for $\mathbf m_2$ equals $0$ since $k_i+t<l_i$. And, of course, the same
happens for $\mathbf m_1 \mathbf m_2$. Therefore,
$${(a)}= \sum_{r=1}^{g(p,i)}\sum_{t=1}^{l_i-k_i-1}\alpha(p,i)^{k_i+t}\beta(p,i,r)\beta(p,i,r)^*(\alpha(p,i)^*)^{k_i+t}.$$
\item For the summands where $k_j < l_j$, an analogous argument works.
\end{enumerate}
Therefore, $(\iota(\mathbf m_1)\vee\iota(\mathbf m_2))\_(\iota(\mathbf m_1\vee\mathbf m_2))=\iota(\mathbf m_1\vee\mathbf m_2)$, whence $\iota(\mathbf m_1\vee \mathbf m_2)\leq \iota(\mathbf m_1)\vee\iota(\mathbf m_2)$.
Since the reverse inequality is obvious, we are done.
\end{proof}

Let us now  introduce some properties for finite covers $\Sigma$ of $\mathcal F_e$.

\begin{definition}
Given $\Sigma\subset\mathcal F_e$ a finite cover, we will say that $\Sigma$ is {\bf orthogonal} if for
all $f\neq g\in\Sigma$, we have that $f$ is orthogonal to $g$, i.e. $fg=0$.
\end{definition}

\begin{proposition}\label{prop:imageiota}
Let $e\in\mathcal E$ and $\Sigma\subset\mathcal F_e$ a finite cover of $\mathcal F _e$. Then there exists an orthogonal finite cover $\widehat\Sigma$ such that
$$\bigvee_{f\in \Sigma}\iota(f)=\bigvee_{g\in\widehat\Sigma}\iota(g).$$
\end{proposition}

\begin{proof}
 We proceed by induction on the cardinality of $\Sigma$. If $|\Sigma | = 1$ , then the result is obvious. Assume that $|\Sigma |>1$ and the result holds for finite covers of cardinality less than $|\Sigma |$.
 Suppose that there are $f,g\in \Sigma $ such that $fg\ne 0$. By induction, it suffices to show that there is a finite cover $\Sigma '$ of $\mathcal F _e$, with $|\Sigma ' |<|\Sigma |$, such that
 $\bigvee_{h\in\Sigma} \iota (h) = \bigvee_{h\in \Sigma'} \iota (h)$. By Lemma~\ref{lem:projections1}, we have that either $f$ and $g$ are comparable, or there is a free prime $p$ such that
 $f=\gamma \mathbf m (p) \gamma ^*$ and $g= \gamma \mathbf m'(p) \gamma^*$ for some c-path $\gamma $ and monomials $\mathbf m (p)$ and $\mathbf m'(p)$ at $p$.
 If $f$ and $g$ are comparable, for instance $f\le g$, then we consider the finite cover $\Sigma '= \Sigma \setminus \{ f \}$. In case $f$ and $g$ are not comparable, then set
 $$\Sigma ' := (\Sigma \setminus \{f,g \} ) \cup \{f' \},$$
 where $f'= \gamma (\mathbf m (p) \vee \mathbf m'(p)) \gamma^*$.  By Lemma~\ref{lem:preservation-suprema}, we have that $\iota (f')= \iota (f) \vee \iota (g)$, and thus
 $\vee_{h\in \Sigma '} \iota (h) = \vee _{h\in \Sigma} \iota (h)$, concluding the proof.
 \end{proof}

We now introduce a crucial concept for our construction.

\begin{definition} \label{def:simple-expansion}
{\rm If $\sum e_i$ is a finite orthogonal sum of idempotents $e _i \in \mathcal E$ in $\mathcal
S_K(E,C)$, then a {\bf simple expansion} of $\sum e_i$ consists of another orthogonal sum of the
form $\sum_{z\in Z} e'_z +  \sum _{i\ne i_0} e_i$, for some $i_0$, obtained by applying the rules
described in Remark~\ref{rk:equalities} to $e_{i_0}$; concretely, the idempotents $e'_z\in
\mathcal E$, $z\in Z$, are of one of the following forms:
\begin{enumerate}
 \item If $e_{i_0}= \gamma \prod_{j=1}^{k(p)} \alpha (p,j)^{k_j}( \alpha(p,j)^*)^{k_j} \gamma^*$ for a free prime $p$, then
 \begin{align*}
\sum_{z\in Z} e_z'  = \gamma \Big[ \alpha(p,j_0)^{k_{j_0}+1}&(\alpha(p,j_0)^*)^{k_{j_0}+1} \prod _{j\ne j_0}  \alpha (p,j)^{k_j}(\alpha (p,j)^*)^{k_j}\Big] \gamma^* \\
& +  \sum_{s=1}^{g(p,j_0)} [\gamma \alpha (p,j_0)^{k_{j_0}}\beta(p,j_0,s)\beta (p,j_0,s)^*(\alpha (p,j_0)^*)^{k_{j_0}}  \gamma^* ]
 \end{align*}
for some $1\le j_0\le k(p)$.
\item If $e_{i_0} = \gamma \lambda \lambda^* \gamma^*$ for some path of finite length $\lambda$ in the graph $E_p$, where $p$ is a regular prime, then $Z= s^{-1}(r(\lambda))$, and
$$ e'_z= \gamma \lambda zz^* \lambda^* \gamma ^* \qquad (z\in Z) .$$
\end{enumerate}}
\end{definition}

\begin{definition} \label{def:expansion}
{\rm An expansion of $e\in \mathcal E$ in $\mathcal S_K(E,C)$ is any expression obtained by applying a finite number of simple expansions to the original expression of $e$.  We define the {\bf expanded set} of the given expression as the set
of orthogonal idempotents in $\mathcal E$ that appear in the expansion.}
\end{definition}

\begin{example}
{\rm  Let $e=\gamma\mathbf m(p)\gamma^*$ be a projection in $\mathcal E$, with $p$ a free prime in
$I$, $\gamma$ a c-path to $v^p$ and $\mathbf m(p)=\alpha(p,i)\alpha(p,i)^*$ for some $i\in
\{1,\ldots, k(p)\}$. Then
\begin{align*}
e &= \gamma\alpha(p,i)\alpha(p,i)^*\gamma^* =\gamma\alpha(p,i)v^p\alpha(p,i)^*\gamma^*\\
  &= \gamma\alpha(p,i)\alpha(p,i)\alpha(p,i)^*\alpha(p,i)^*\gamma^*+\gamma\alpha(p,i)\left[\sum_{s=1}^{g(p,i)}\beta(p,i,s)\beta(p,i,s)^*\right]\alpha(p,i)^*\gamma^*\\
  &= \gamma\alpha(p,i)\alpha(p,i)\alpha(p,i)^*\alpha(p,i)^*\gamma^*+\sum_{s=1}^{g(p,i)}[\gamma\alpha(p,i)\beta(p,i,s)\beta(p,i,s)^*\alpha(p,i)^*\gamma^*].
\end{align*}
The latter equation is what we call an {\bf expansion} of $e$. The {\bf expanded set} of the above
expression contains the projections
$\gamma\alpha(p,i)\alpha(p,i)\alpha(p,i)^*\alpha(p,i)^*\gamma^*$ and also each of the projections
$\gamma\alpha(p,i)\beta(p,i,s)\beta(p,i,s)^*\alpha(p,i)^*\gamma^*$.}
\end{example}

A consequence of the above definition is the following.

\begin{lemma}\label{lem:expanded}
If $\Sigma$ is an expanded set of $e\in\mathcal E$, then it is an orthogonal finite cover of $\mathcal F_e$.
\end{lemma}
\begin{proof}
If $\Sigma$ is an expanded set of $e\in\mathcal E$, then $\Sigma\subset\mathcal F_e$, and its
elements are pairwise orthogonal. We will show that for any idempotent $f\le e$ and any expanded
set $\Sigma $ of $e$, either $f\le g$ for some $g\in \Sigma $ or $g\le f$ for some $g\in \Sigma$.
We proceed by induction on the number of simple expansions (Definition~\ref{def:simple-expansion})
used to get $\Sigma $ from $e$. If the number of simple expansions is $0$, then $\Sigma = \{ e \}$
and $f\le e$, as desired. Assume the result holds for expansions obtained by using $n-1$ simple
expansions of $e$, for $n\ge 1$, and let $\Sigma $ be an expansion of $e$ obtained by using $n$
simple expansions. Then there exists an expansion $\Sigma '$ obtained from $e$ by using $n-1$
simple expansions and an element $g'$ of $\Sigma '$ such that
$$\Sigma = (\Sigma ' \setminus \{ g' \}) \cup \{g_1,g_2, \dots , g_N \}.$$
Indeed, write $g'= \gamma \mathbf m (p) \gamma^*$ in its standard form. If $p$ is a free prime, then write
$g'= \gamma \prod_{j=1}^{k} \alpha (p,j)^{k_j}(\alpha (p,j)^*)^{k_j} \gamma ^*$ for $k_j\in\Z^+$, and, for some $1\le i\le k(p)$, we have
$$g_1= \gamma \alpha (p,i)^{k_i+1} (\alpha (p,i)^*)^{k_i+1}\prod_{j\ne i} \alpha(p,j)^{k_j}(\alpha (p,j)^*)^{k_j} \gamma ^*$$ and
$$g_{s+1} = \gamma \alpha (p,i)^{k_i} \beta (p,i,s) \beta (p,i,s)^* (\alpha (p,i)^*)^{k_i} \gamma^*,$$
for $1\le s \le g(p,i)$, where $N=g(p,i)+1$. If $p$ is a regular prime, one has
$$g'= \gamma \lambda \lambda^* \gamma^* ,$$
where $\lambda$  is a path of finite length in the graph $E_p$, and
$$g_i = \gamma \lambda e_ie_i^*\lambda ^* \gamma ^*,$$
with $s^{-1}(r(\lambda)) = \{ e_1, \dots , e_N \}$.

By the induction hypothesis, there is some $g''\in \Sigma '$ such that either $f\le g''$ or
$g''\le f$. Now, if $g'\ne g''$, then we take $g=g''$ and either $f\le g$ or $g\le f$. So assume
that $g'=g''$. If $g'\le f$ then clearly $g_i\le f$ for $i=1,2, \dots , N$, so we can take $g$ to
be any of them. Finally assume that $f < g'$. When $p$ is a free prime, this means that $f= \gamma
\gamma ' \mathbf n (p') (\gamma ')^* \gamma^*$, where $\gamma '$ is a c-path from $v^p$ to $w\in
E_{p'}^0$. If $\gamma '$ is a trivial c-path, then $\mathbf n (p')=\prod \alpha(p,j)^{l_j}(\alpha
(p,j)^*)^{l_j}$ and necessarily we have that $l_j\ge k_j$ for all $j$. Now if $l_i\ge k_i+1$, then
we get $f\le g_1$, so we can take $g=g_1$. If $l_i = k_i$, then $g_{s+1}\le f$ for $s=1,\dots,
g(p,i)$, so we can take $g$ to be any of them. Assume now that $\gamma '$ is non-trivial and write
$\gamma ' = \alpha (p,i')^{m} \beta (p, i', s) \cdots $, for some $i'\in \{ 1, \dots , k(p) \}$
and some $s\in \{ 1,\dots, g(p,i') \}$. Since $f\le g'$, we must have $m\ge k_{i'}$. Now if $i\ne
i'$, then we will have $f\le g_1$, and we take $g=g_1$. If $i= i'$ and $m=k_i$, then $f\le
g_{s+1}$ and we take $g=g_{s+1}$. Finally, if $i=i'$ and $m
> k_i$, then $f\leq g_1$ and we take $g=g_1$.

If $p$ is a regular prime, then since $f < g'$, there exists $e_i\in s^{-1}(r(\lambda))$ such that
$f\le \gamma \lambda e_ie_i^*\lambda^*\gamma ^* = g_i$.
\end{proof}

The next result shows that the converse of Lemma~\ref{lem:expanded} also holds.





\begin{proposition}\label{prop:converseExpanded}
Let $e\in\mathcal E$, and let $\Sigma\subset\mathcal F_e$ be an orthogonal finite cover of $\mathcal F_e$. Then, $\Sigma$ is an expanded set of $e$.
\end{proposition}
\begin{proof}
We will proceed by induction on the number of elements
of $\Sigma$. The result is clear if $| \Sigma | = 1$. Assume that  $| \Sigma | >1$ and the result holds for finite orthogonal covers of idempotents in $\mathcal E$ of cardinality less than $|\Sigma |$.

Suppose first that $e=\gamma\prod _{i=1}^{k(p)} \alpha (p,i)^{k_i}(\alpha (p,i)^*)^{k_i}
\gamma^*$, where $p$ is a free prime. We first show that $\Sigma $ must contain some element $e'$
of the form $e' = \gamma\prod _{i=1}^{k(p)} \alpha (p,i)^{l_i}(\alpha (p,i)^*)^{l_i} \gamma^*$,
where $l_i>k_i$ for at least one index $i$. Assume not. Then each $f\in \Sigma $ is of the form
$\gamma \gamma ' \mathbf m (\gamma')^*\gamma^*$, where $\gamma '= \alpha (p, i(f))^{ m(f)} \beta
(p,i(f), s(f))\cdots $ for some non-negative integers $m(f)$ and positive integers $s(f)$. Let
$\Pi = \{ i(f) \mid f\in \Sigma \} \subseteq \{1,\dots ,k(p) \}$  be the set of indices appearing
in the starting terms of the elements of $\Sigma$. For each $i\in \Pi$,  take an integer $M_i$
such that $M_i >m(f)$ for every $f\in \Sigma $ such that $i(f) = i$. Note that $M_i>k_i$ for $i\in
\Pi$. If $i\notin \Pi$, set $M_i= k_i$. Now consider $f\in \mathcal E$ defined by
$$f= \gamma \prod _{i=1}^{k(p)} \alpha (p,i)^{M_i}(\alpha (p,i)^*)^{M_i} \gamma^*.$$
Then $0\ne f\in \mathcal F_e$ and $fg= 0$ for each $g\in \Sigma$,
which is a contradiction. Therefore there is an element $e'$ in
$\Sigma$ of the form described above, and it is necessarily unique,
because $\Sigma $ is orthogonal, and the product of idempotents of
this form is always nonzero. Note that $e'\ne e$ because $|\Sigma
|>1$ and thus there is at least one index $i$ such that $l_i>k_i$.

Now we consider the following expansion of $e$ in $\mathcal
S_K(E,C)$ leading to $e'$:
$$e= e'+ \sum_{k_i<l_i}\sum_{t=0}^{l_i-k_i-1} \sum_{s=1}^{g(p,i)} \alpha (p,i)^{k_i+t}\beta (p,i,s)\beta (p,i,s)^*(\alpha (p,i)^*)^{k_i+t}.$$
Let $$\Omega = \{ \alpha (p,i)^{k_i+t}\beta (p,i,s)\beta (p,i,s)^*(\alpha (p,i)^*)^{k_i+t} \mid
k_i<l_i, 0\le t\le l_i-k_i-1 , 1\le s\le g(p,i) \} $$ and observe that $\Omega \cup \{ e' \}$ is
an expanded set of $e$. The proof of Lemma~\ref{lem:expanded} shows that for any $g\in \Sigma
\setminus \{ e'\}$ there exists $f\in \Omega $ such that either $g\le f$ of $f\le g$. But now if
$f\le g$ for $f\in \Omega $ and $g\in \Sigma \setminus  \{ e' \}$ then necessarily $f=g$. Thus, in
any case, for each $g\in \Sigma \setminus \{ e' \}$ there is $f\in \Omega $ such that $g\le f$.
Since the elements of $\Omega $ are pairwise orthogonal, there is exactly one $f\in \Omega $ such
that $g\le f$.

For $f\in \Omega $, let $\Sigma_f = \{ g\in \Sigma \setminus \{ e' \} \mid g\le f \}$. By the
above argument, $\{ \Sigma_f \mid f\in \Omega \}$ is a partition of $\Sigma \setminus \{ e' \} $,
Moreover, it follows easily that $\Sigma _f$ is an orthogonal finite cover of $f$ for each $f\in
\Omega $. By the induction hypothesis, $\Sigma _f$ is an expanded set of $f$, for $f\in \Omega $.
Hence $\Sigma $ is an expanded set of $e$. This shows the result in the case where $p$ is a free
prime.

Assume now that $p$ is a regular prime, and write $e= \gamma \lambda \lambda ^* \gamma^*$, where
$\gamma $ is a c-path and $\lambda $ is a path of finite length in the graph $E_p$. By a similar
argument as before, there is $e'\in \Sigma $ of the form $e'= \gamma \lambda \lambda'
(\lambda')^*\lambda^* \gamma ^*$, where $\lambda'$ is a path of finite length in $E_p$. We assume
that $\lambda'$ has maximal length, say $r$, amongst all the paths in $E_p$ giving rise to an
element of $\Sigma$. Write $\lambda ' = e_1e_2\cdots e_r$, where $e_i\in E_p^1$. We consider the
following expansion of $e$ in $\mathcal S_K(E,C)$:
$$e= e' +\sum _{i=0}^{r-1} \sum_{ f\ne e_{i+1}, s(f) =s(e_{i+1})} \gamma \lambda e_1\cdots e_i ff^* e_i^*\cdots e_1^* \lambda^* \gamma^* .$$
Let
$$\Omega = \{ \gamma \lambda e_1\cdots e_i ff^* e_i^*\cdots e_1^* \lambda^* \gamma^* \mid i=0,\dots , r-1, f\in E^1, f\ne e_{i+1}, s(f)= s(e_{i+1}) \} .$$
Then $\Omega \cup \{ e' \}$ is an expanded set of $e$. Therefore, by the proof of
Lemma~\ref{lem:expanded}, for any $g\in \Sigma \setminus \{ e'\}$ there exists $h\in \Omega $ such
that either $g\le h$ or $h\le g$. If $h\le g$, where $h= \gamma \lambda e_1\cdots e_i ff^*
e_i^*\cdots e_1^* \lambda^* \gamma^*$ and $h\ne g$, then $g= \gamma \lambda e_1\cdots e_j
e_j^*\cdots e_1^*\lambda^* \gamma^*$ for some $j\le i$, forcing $ge'\ne 0$, and contradicting
irredundancy of $\Sigma$. Thus, in any case, for each $g\in \Sigma \setminus \{ e' \}$ there is
$h\in \Omega $ such that $g\le h$. Now the proof ends exactly as in the case where $p$ is a free
prime.
\end{proof}

Now, we are ready to show that the map $\iota$ is tight.

\begin{theorem}\label{thm:tight}
The map $\iota:S(E,C) \to \mathcal S_K(E,C)$ is tight.
\end{theorem}

\begin{proof}
Let $e\in\mathcal E$, and $\Sigma\subset \mathcal F_e$ be a finite cover. Then, by
Proposition~\ref{prop:imageiota}, there exists an orthogonal finite cover $\widehat\Sigma$ of
$\mathcal F_e$ such that $$\bigvee_{f\in\Sigma}\iota(f)=\bigvee_{g\in\widehat\Sigma}\iota(g).$$ By
Proposition~\ref{prop:converseExpanded}, $\widehat\Sigma$ is an expanded set of an expression of $e$
in $\mathcal S_K(E,C)$, whence $$\iota(e)=\sum_{g\in\widehat\Sigma}\iota(g)\quad \text{ in }\quad
\mathcal S_K(E,C).$$ Because the elements in $\widehat\Sigma$ are pairwise orthogonal,
$$\sum_{g\in\widehat\Sigma}\iota(g)=\bigvee_{g\in\widehat\Sigma}\iota(g).$$ Hence
\[
\iota(e)=\sum_{g\in\widehat\Sigma}\iota(g)=\bigvee_{g\in\widehat\Sigma}\iota(g)=\bigvee_{f\in\Sigma}\iota(f).\qedhere
\]
\end{proof}


\subsection{Universality of the Representation} In this subsection, we will show that the tight representation $\iota$ of $S(E,C)$ described above is universal for tight representations.
\begin{theorem}\label{thm:Universal}
The map $\iota:S(E,C) \to \mathcal S_K(E,C)$ is universal for tight representations of $S(E,C)$ on $*$-algebras over $K$.
\end{theorem}

\begin{proof}
Let $A$ be a $*$-algebra, $\phi: S(E,C) \to A$ be a tight representation and denote by
$\overline{x}$ the images of the elements of $S(E,C)$ under $\phi$. Since $\phi$ is tight,
Definition~\ref{def:tight} and Lemma~\ref{lem:expanded} show that
\begin{enumerate}
\item for each $p\in I_{reg}$ and each $\in E^0_p$, $$\ol{v}=\sum_{e\in s^{-1}(v)}\ol{e} \cdot
    \ol{e}^*,\text{ and}$$
\item for each $p\in I_{free}$ and each $1\leq i\leq k(p)$,
    $$\ol{v^p}=\ol{\alpha(p,i)} \cdot \ol{\alpha(p,i)}^*+\sum^{g(p,i)}_{s=1}\ol{\beta(p,i,s)}\cdot \ol{\beta(p,i,s)}^*.$$
\end{enumerate}
All other relations in $\mathcal S_K(E,C)$ are preserved because $\phi$ is a representation. Thus
there exists a unique $K$-algebra $*$-homomorphism $\Phi:\mathcal S_K(E,C)\to A$ such that
$\Phi(\iota(x))=\phi(x)$ for each generator $x$ of $S(E,C)$.
\end{proof}

Using Theorem~\ref{thm:Universal} and the results of Section~\ref{sect:Steinbergalges}, we obtain
an isomorphism between the $K$-algebra $\mathcal S_K(E,C)$ and the Steinberg algebra $A_K
(\mathcal G_{tight} (S(E,C)))$.

\begin{theorem} \label{thm:Steinbergalgebra-and-our-algebra}
Let $(E,C)$ be an adaptable separated graph, let $S(E,C)$ be the inverse semigroup associated to
$(E,C)$, let $K$ be a field with involution and let $\mathcal S_K(E,C)$ be the $*$-algebra over
$K$ associated to $(E,C)$. Let $A_K (\mathcal G _{tight} (S(E,C)))$ be the Steinberg algebra of
the tight groupoid $\mathcal G_{tight} (S(E,C))$. There is a $*$-isomorphism
\[
\mathcal S_K(E,C) \cong A_K(\mathcal G_{tight} (S(E,C)))
\]
sending $\iota (s)\in \mathcal S_K(E,C)$ to $1_{\Theta (s, D_{s^*s})}$ for each $s\in S(E,C)$.
\end{theorem}
\begin{proof}
Since $S(E,C)$ is $E^*$-unitary (Proposition~\ref{prop:E-star-unitary-general}), it is a Hausdorff
inverse semigroup. Hence the result follows from Theorem~\ref{thm:steinberg16} and Theorems
\ref{thm:tight}~and~\ref{thm:Universal}.
\end{proof}

\begin{corollary}\label{cor:Groupoid-C*-algebra-and-our-algebra}
Let $(E,C)$ be an adaptable separated graph,  $S(E,C)$ be its associated inverse semigroup, and $C^*(S(E, C))$ be the universal $C^*$-algebra defined by generators $E^0\cup E^1\cup \{
(t_i^v)^{\pm 1} : i\in \N,  v\in E^0 \}$  and the relations~{\bf\ref{pt:KeyDefs}}. Then there is an
isomorphism
\[
C^*(S(E,C)) \cong C^*(\mathcal G_{tight} (S(E,C)))
\]
sending $\iota(s)\in C^*(S(E,C))$ to $1_{\Theta (s, D_{s^*s})}$ for each $s\in S(E,C)$.
\end{corollary}
\begin{proof}
The proofs of \cite[Theorem~3.2.2 and Lemma~3.2.3]{SimsCRMNotes} are easily modified to show that
$C^*(\mathcal G_{tight}(S(E, C))$ is universal for $^*$-representations of
$A_{\mathbb{C}}(\mathcal{G}_{tight}(S(E, C)))$. By definition, $C^*(S(E, C))$ is universal for
$^*$-representations of $\mathcal{S}_{\mathbb{C}}(E,C)$. So the result follows from
Theorem~\ref{thm:Steinbergalgebra-and-our-algebra}.
\end{proof}

\section{A new description of the groupoid \texorpdfstring{$\mathcal G_{tight}(S(E,C))$}{Gtight(S(E,C))}}\label{sect:description-groupoid}

In this section, we obtain a more concrete description of the tight groupoid $\mathcal
G_{tight}(S(E,C))$ defined in the last section, paralleling the well-known description of graph
groupoids (see, for example, \cite{CFST}).

To this end, let $\widehat{\mathcal E}_{\infty}$ be the space of infinite paths, described above (see
Section~\ref{sect:filters}), and consider the additive group $\Gamma := \Z^{(\infty )}\times
\Z^{(\infty)}$. Here, the first copy of $\Z^{(\infty )}$ will encode the exponents of the
indeterminates $t^v_i$ and the second copy of $\Z^{(\infty )}$ will encode the difference of
lengths between initial segments of infinite paths.

Our new approach to $\mathcal G_{tight}(S(E,C))$ is based on understanding the groupoid we
subsequently define. First, if $\gamma$ is an initial segment ending in $E_p$, we define
$|\gamma|_\infty\in \Z^{(\infty )}$ as follows: if $p$ is a regular, then $|\gamma|_{\infty} :=
(|\gamma |, 0,0, \dots, )\in \Z^{(\infty )}$; and, if $p$ is a free prime, then, expressing
$\gamma  = \gamma ' \prod _{j=1}^{k(p)} \alpha (p,j)^{l_j}$ where $\gamma '$ is a c-path and
$l_j\in \Z^+$, then $|\gamma |_{\infty} := (l_1+ |\gamma '|, \dots, l_{k(p)}+ |\gamma ' |,
0,0,\dots )$.

Now consider the set $\mathcal H$ of triples $( x , n , y) $ such that
\begin{enumerate}[(i)]
\item $x$ and $y$ are infinite paths ending at the same prime $p$ of $I$; say $x= \gamma \lambda
    $, and $y= \nu \lambda$, where $\gamma $ and $\nu$ are initial segments ending in $E_p$, and
    $\lambda $ is an infinite path in $E_p$ (see Definitions
    \ref{def:semfinite-path}~and~\ref{def:initial-segment}); and
\item $n= (n_1, |\gamma|_{\infty} - |\nu|_ {\infty}) \in \Gamma$ for some $n_1 \in
    \Z^{(\infty)}$.
\end{enumerate}

We define $\mathcal{H}^{(2)} = \{((w, m, x), (y, n, z)) \in \mathcal{H} \times \mathcal{H} : x =
y\}$, and we define multiplication from $\mathcal{H}^{(2)}$ to $\mathcal{H}$ by
$$(x,n,y)(y,m,z)= (x,n+m,z).$$
Then $\mathcal{H}$ is a groupoid, with inverses given by $(x,n,y)^{-1} = (y,-n,x)$, and units
$\mathcal H ^{(0)} = \{(x, 0, x) : x \in \widehat{\mathcal E}_{\infty}\}$ identified with
$\widehat{\mathcal E}_{\infty}$.

We now introduce a topology in $\mathcal H$. A basis of this topology is indexed by the elements in the inverse semigroup $S(E,C)$.

\begin{definition}\label{Def 7.1}
{\rm Given $s \in S(E,C)$, written in the standard notation as  $s = \gamma \mathbf m (p) \nu^*$,
for c-paths $\gamma ,\nu$ and monomial $ \mathbf m (p)$ at $p\in I$, define $\mathcal Z (s)$ as
the set of elements $  (x,n,y) \in \mathcal H$ which satisfy the following conditions:
\begin{enumerate}
 \item If $p$ is a regular prime, then $\mathbf m (p) = \prod (t_i^v)^{d_i} \lambda \eta^* $,
     for paths of finite length $\lambda, \eta $ in $E_p$, with $s(\lambda ) =v\in E_p^0$ and
     $r(\lambda ) =r(\eta)$. In this case, $x= \gamma \lambda x_0$ and $y= \nu \eta x_0$, for an
     infinite path $x_0= \gamma ' \lambda'$  starting at $r(\lambda)$, where $\gamma '$ is a
     c-path and $\lambda'$ is an infinite path in the component corresponding to $r(\gamma ')$.
     Moreover, $n=(n_1,n_2)$, where $n_1\in \Z^{(\infty )}$ is the sequence of exponents of
     $\phi_{\gamma '}(\mathbf m (p))$  and $n_2 = | \gamma \lambda \gamma '|_{\infty} -
     |\nu\eta\gamma' |_{\infty} \in \Z^{(\infty )}$. (Observe that $n_2$ is determined by the
     integer $|\gamma | +|\lambda |-|\nu | - |\eta |$).
 \item If $p$ is a free prime, then $\mathbf m (p)=   \prod (t_i^{v^p})^{d_i} \prod_{j=1}^{k(p)}
     \alpha (p,j)^{k_j}(\alpha (p,j)^*)^{l_j}$. In this case, either $x= \gamma \prod
     _{j=1}^{k(p)} \alpha (p,j)^{\infty}$,  $y= \nu \prod_{j=1}^{\infty} \alpha(p,j)^{\infty}$,
     and $n=(n_1,n_2)$, with $n_1= (d_i)_{i\in \N}$, $n_2= (k_1 + |\gamma | -l_1-|\nu|, \dots ,
     k_{k(p)} + |\gamma | -l_{k(p)} - |\nu|, 0,\dots )$, or there exists an infinite path
     $\gamma ' \lambda '$, where $\gamma '= \alpha (p,i)^m \beta (p,i,s) \cdots $, $1\le i \le
     k(p)$  is a non-trivial c-path and $\lambda '$ is an infinite path in the component
     corresponding to $r(\gamma ')$, such that
 $$x= \gamma \alpha (p,i)^{k_i} \gamma' \lambda ', \quad y= \nu \alpha (p,i)^{l_i} \gamma ' \lambda ' ,$$
 and $n=(n_1,n_2)$, with $n_1$ equal to the sequence of exponents of $\phi_{\gamma'} ( \mathbf m
 (p))$ and $n_2= |\gamma \alpha (p,i)^{k_i}\gamma '|_{\infty} - |\nu \alpha(p,i)^{l_i} \gamma
 '|_{\infty}\in \Z^{(\infty )}$. (Note that $n_2$ is determined by the integer $ |\gamma | -
 |\nu | + k_i - l_i$).
 \end{enumerate}
 }
 \end{definition}

 \begin{remark}
{\rm  Notice that when $e=\gamma \mathbf m(p)\gamma^*$ is an idempotent in $S(E,C)$, then
 \begin{enumerate}[\rm (a)]
 \item if $p$ is a free prime, then $\mathbf m (p) = \prod_{j=1}^{k(p)} \alpha
     (p,j)^{k_j}(\alpha (p,j)^*)^{k_j}$, and  $\mathcal Z (e)$ is the set of elements
     $(x,(0,0),x)\in \mathcal H$ such that $\lambda \prod_{j=1}^{k(p)} \alpha(p,j)^{k_j}$ is an
     initial segment of $x$; and
 \item if $p$ is a regular prime, then $\mathbf m (p) = \lambda \lambda ^*$ for a path of finite length $\lambda $ in the graph $E_p$, and $\mathcal Z (e)$
 is the set of elements $(x,(0,0),x)\in \mathcal H$ such that $\gamma \lambda $ is an initial segment of $x$.
 \end{enumerate}
 Therefore, using the identification of ultrafilters with infinite paths given in Corollary~\ref{cor:ultrafilters}, we notice that $\mathcal Z (e)$ may be identified
 with the set $\mathcal Z(e)$ of Definition~\ref{def:familycalP}.}
 \end{remark}

 We now establish the isomorphism between the groupoid of germs of the action of $S(E,C)$ on $\widehat{\mathcal E}_{\infty}$ and the groupoid $\mathcal H$.

\begin{theorem}
 \label{thm:iso-groupoids}
 Let $(E,C)$ be an adaptable separated graph. Then there is a canonical isomorphism
  $$\phi \colon \mathcal G_{tight}(S(E,C)) \to \mathcal H$$
 between the topological groupoid $\mathcal G_{tight}(S(E,C))$ and
 the topological groupoid $\mathcal H$ defined above.
 \end{theorem}

\begin{proof}
 We start by observing that the elements of $\mathcal G_{tight}(S(E,C))$ can be represented by classes $[s,x]$, where $(s,x)$ is of a special form. Recall that $x$ is a point in the space $\widehat{\mathcal E}_{\infty}$, and that $s$ is an element in $S(E,C)$ such that $s^*s\in x$ (interpreting $x$ as a subset of $\mathcal E$). Let us describe the structure of $x$ by cases:

 \begin{enumerate}
 \item Assume first that $x= \nu \lambda$, where $\nu $ is a c-path and $\lambda $ is an
     infinite path in the component of a regular prime $p$. Then, we can assume that $s= \gamma
     (\prod (t^v_i)^{d_i})  \lambda ' (\lambda'')^* \nu^*$, where $\gamma$ is a c-path,
     $r(\gamma ) = s(\lambda ') = v\in E_p^0$, $r(\lambda ')= r(\lambda '')$ and $\lambda ''$ is
     an initial segment of $\lambda$, that is, $\lambda = \lambda '' x_0$ for an infinite path
     $x_0$. In this case, we set
 $$\phi ([s,x]) = ( \gamma \lambda' x_0, ((d_i)_{i\in \N}, (|\gamma | + |\lambda '|-|\lambda''|- |\nu|,0,\dots )), \nu\lambda''x_0 ) \in \mathcal H .$$
 \item Assume now that $x= \nu  \prod_{j=1}^{k(p)} \alpha (p,j)^{\infty}$ for a free prime $p$,
     where $\nu $ is a c-path ending at $v^p$. Then, $s$ can be chosen of the form
\begin{equation}\label{eqn:s}
s = \gamma \Big(\prod (t^{v^p}_i)^{d_i}\Big)\prod_{j=1}^{k(p)} \alpha (p,j)^{k_j}(\alpha (p,j)^*)^{l_j} \nu^*,
\end{equation}
where $\gamma $ is a c-path ending at $v^p$. For an element $[s,x]$ of this form, we define
$$\phi ([s,x])= \Big(\gamma \prod_{j=1}^{k(p)} \alpha (p,j)^{\infty}, ((d_i)_{i\in \N}, (|\gamma | + k_j-l_j -|\nu |)_j), \nu \prod_{j=1}^{k(p)} \alpha (p,j)^{\infty}\Big)\in \mathcal H.$$
 \end{enumerate}

In both cases (1)~and~(2), we can achieve the representative of the desired form by replacing the
given element $[s,x]$ with a suitable element of the form $[sf,x]$, where $f\in \mathcal E$.

Notice that $\phi$ is a well-defined map. Indeed, if $[s,x]=[t,y]$, then $x=y$ and there exists
$f\in x$ (thinking $x$ as a subset of $\widehat{\mathcal{E}}_{\infty}$) such that $sf=tf$. In the free
case, we can assume that $f=\nu \prod \alpha (p,j)^{t_j} (\alpha(p,j)^*)^{t_j} \nu^*$, where $t_j$
is bigger or equal than the corresponding powers of $\alpha (p,j)^*$ in $s$ and $t$. Hence, the
elements $[s,x]$ and $[t,x]$ have the same image under $\phi$. The regular case is analogous.

It is straightforward to check that $\phi$ is an algebraic morphism of groupoids. Moreover, by the
definition of $\mathcal H$, $\phi$ is surjective. Let us check that $\phi$ is injective. To this
end, let us firstly work the free prime case out. Assume that $x= \nu  \prod_{j=1}^{k(p)} \alpha
(p,j)^{\infty}$ for a free prime $p$ and, fixing notation, that $s, t \in S(E,C)$,  with $s$ as
in~(\ref{eqn:s}) and
$$t = \gamma' \Big(\prod (t^{v^p}_i)^{d_i'}\Big)\prod_{j=1}^{k(p)} \alpha (p,j)^{k_j'}(\alpha (p,j)^*)^{l_j'} \nu^*.$$
If $\phi ([s,x])= \phi ([t,x])$, then $\gamma = \gamma '$, $d_i=d_i'$ for all $i$ and $k_j-l_j= k_j'-l_j'$ for all $1\le j\le k(p)$. But then $k_j+l_j'= k_j'+l_j$ for all $j$, and
so the element $f= \nu \prod_{j=1}^{k(p)} \alpha (p,j)^{k_j+l_j'}(\alpha (p,j)^*)^{k_j'+l_j} \nu^*$ is a idempotent in $S(E,C)$ such that $f\in x$ and $sf=tf$, which implies that $[s,x]= [t,x]$.
The case where $x$ is an infinite path ending in a regular prime is left to the reader. Therefore, we conclude that $\phi$ is a bijective isomorphism of algebraic groupoids.

It remains only to show that the map $\phi $ is a homeomorphism. Recall that the topology of
$\mathcal G : = \mathcal G_{tight}(S(E,C))$ is generated by the sets
$$ \Theta (s,U) := \{ [s,x] \in \mathcal G  \mid x\in U \},$$
where $U$ is an open subset of $\widehat{\mathcal E}_{\infty}$ such that $U\subseteq D_{s^*s} := \{
x\in \widehat{\mathcal E}_{\infty} \mid s^*s\in x \}$. Given such an open set $U$, we know from
Corollary~\ref{cor:compact-opens} that $U= \bigcup_{f\in Z_U} \mathcal Z (f)$, where $Z_U$ is the
set of idempotents $f\in \mathcal E$ such that $\mathcal Z (f) \subseteq U$, and thus
$$\Theta (s, U) = \bigcup_{f\in Z_U} \Theta (sf, \mathcal Z (f)).$$
Therefore, the topology of $\mathcal G$ is generated by the open sets
$$U_s : = \Theta (s, \mathcal Z (s^*s)) =  \{ [s,x] \mid s^*s\in x \},$$
for $s\in S$. It is straightforward to check that $\phi (U_s)= \mathcal Z (s)$ for each $s\in S(E,C)$. This shows that $\phi$ is a homeomorphism, and the proof is complete.
\end{proof}

Summing up, we obtain the following result.

\begin{theorem}
 \label{thm:gordo1}
 Let $(E,C)$ be an adaptable separated graph, and $S(E,C)$ be the inverse semigroup associated to $(E,C)$.
 Let $\mathcal G_{tight}(S(E,C))$ be the groupoid of germs associated to the canonical action of $S(E,C)$ on the space of ultrafilters $\widehat{\mathcal E}_{\infty}$, and let $\mathcal H$ be the groupoid
 defined above. Then $\mathcal G_{tight}(S(E,C)) \cong \mathcal H$ is an ample Hausdorff \'etale topological groupoid.
\end{theorem}

\begin{proof}
 By Theorem~\ref{thm:iso-groupoids}, there is an isomorphism of topological groupoids $\mathcal G_{tight}(S(E,C)) \cong \mathcal H$. By \cite[Proposition 4.17]{ExelBraz}, the groupoid
 $\mathcal G_{tight}(S(E,C))$ is an \'etale topological groupoid. Since $S(E,C)$ is $E^*$-unitary by Proposition~\ref{prop:E-star-unitary-general}, it follows from \cite[Propositions 6.4 and 6.2]{ExelBraz}
 that $\mathcal G_{tight}(S(E,C))$ is Hausdorff. Finally, observe that $U_s = \Theta (s, \mathcal Z (s^*s))$, $s\in S(E,C)$, is a basis of compact open bisections for the topology of $\mathcal G_{tight}(S(E,C))$.
 \end{proof}

\section{Amenability of \texorpdfstring{$\mathcal G_{tight}(S(E,C))$}{Gtight(S(E,C))}}\label{Sect:Amenability}

In this section we show that the groupoid $\Gg_{tight}(S(E,C))$ associated to any adaptable
separated graph is amenable (Proposition~\ref{prop:amenable}). Thus the full and reduced
$C^*$-algebras of this groupoid coincide, are nuclear, and belong to the UCT class. In particular,
Corollary~\ref{cor:Groupoid-C*-algebra-and-our-algebra} shows that $C_r^*(\mathcal G_{tight}(S(E,
C))$ is universal for tight representations of $S(E, C)$, so $C^*(S(E, C))$ has a faithful
representation on $\bigoplus_{x \in \widehat{\mathcal{E}}^\infty} \ell^2(G_{tight}(S(E, C))_x)$.
These results open the possibility of addressing the realisation problem via nuclear
$C^*$-algebras in the UCT class.
%
%

The first step is to reduce the question to dealing with the groupoid sitting over any prime $p \in I$, because this significantly simplifies the question. To do this, we will
use the following general lemma, together with the fact that, since $I$ is finite, the
groupoid admits a finite ``composition series." In the following result, given a poset $\Oo$ and an element $V \in \Oo$ we will say that
$U \in \Oo$ is \emph{subjacent} to $V$ if $U$ is a maximal element of the set $\{W \in
\Oo \setminus \{V\}: W \preceq V\}$.

\begin{lemma}\label{lem:composition series}
    Let $\Gg$ be a second-countable Hausdorff \'etale groupoid. Suppose that $\Oo$ is a
    finite collection of open invariant subsets of $\Gg^{(0)}$ that is closed under
    intersections. Suppose that $\Gg^{(0)}$ and $\emptyset$ both belong to $\Oo$, and regard
    $\Oo$ as a poset under $\subseteq$. Suppose that, whenever $U$ is subjacent to $V$ in
    $\Oo$, the groupoid $\Gg|_{V \setminus U}$ is amenable. Then, $\Gg$ is amenable.
\end{lemma}
\begin{proof}
    We proceed by induction on $|\Oo|$. First, suppose that $|\Oo| = 1$. Since $\Gg^{(0)}$ and
    $\emptyset$ both belong to $\Oo$, we deduce that $\Gg^{(0)} = \emptyset$, and so $\Gg =
    \emptyset$ is trivially amenable.

    Now fix $n \ge 1$. Suppose as an inductive hypothesis that the result holds for groupoids
    in which $|\Oo| \le n$, and suppose that $|\Oo| = n+1$. Let $V$ be a maximal element of
    $\Oo \setminus \{\Gg^{(0)}\}$. By hypothesis, $\Gg|_{\Gg^{(0)} \setminus V}$ is amenable.
    Let $\Oo' := \{U \cap V : U \in \Oo\}$. Since $\Oo$ is closed under intersections, we
    have $\Oo' = \{W \in \Oo : W \subseteq V\} \subseteq \Oo$, and in particular if $U$ is
    subjacent to $U'$ in $\Oo'$ then $U$ is subjacent to $U'$ in $\Oo$. So $\Oo'$ is a finite
    collection of open invariant subsets of $\Gg|_V^{(0)}$ that is closed under intersections
    and contains $\Gg|_V^{(0)}$ and $\emptyset$, and whenever $U$ is subjacent to $U'$ in
    $\Oo'$ the groupoid $(\Gg|_V)_{U' \setminus U} = \Gg|_{U' \setminus U}$ is amenable.
    Since $V \cap V = \Gg^{(0)} \cap V$, we see that $|\Oo'| \le |\Oo| - 1 = n$. So the
    inductive hypothesis implies that $\Gg|_V$ is amenable. Thus
    \cite[Proposition~9.82]{WilliamsGpds} shows that $\Gg$ is amenable.
\end{proof}

\begin{proposition}\label{prop:amenable}
    Let $(E,C)$ be an adaptable separated graph. The
    groupoid $\Gg_{tight}(S(E,C))$ is amenable.
\end{proposition}

\begin{proof}
    Let $\Hh$ be the groupoid, isomorphic to $\Gg_{tight}(S(E,C))$, described in Section~\ref{sect:description-groupoid}.
    We say that a subset $P \subseteq I$ is \emph{hereditary} if $p
    \in P$ and $p \ge q$ implies $q \in P$. For each hereditary $P \subseteq I$, the set
    \begin{align*}\textstyle
        W_P := \bigcup_{p \in P} \{\gamma\lambda \in \Hh^{(0)} : \gamma &\text{ is an initial segment}, r(\gamma)\in E_p^0\text{ and }\lambda\in E_p^{\infty}\}
    \end{align*}
    is an open invariant subset of $\Hh^{(0)}$. Clearly, the collection $\Oo = \{W_P : P
    \subseteq I\text{ is hereditary}\}$ is a finite collection of open invariant sets that
    contains $\emptyset = W_\emptyset$ and $\Hh^{(0)} = W_I$, and is closed under
    intersections since $W_P \cap W_Q = W_{P \cap Q}$.

    If $P, Q$ are hereditary subsets of $I$, then $W_P \subseteq W_Q$ if and only if $P
    \subseteq Q$. Suppose that $U,V \in \Oo$, and that $U$ is subjacent to $V$ in $\Oo$.
    Then $V = W_Q$ and $U = W_P$ for some nested pair $P
    \subsetneq Q$ of hereditary subsets of $I$. Since $Q \setminus P$ is nonempty, we deduce that
    it contains a maximal element $q \in Q$. Hence $Q' = Q \setminus \{q\}$ is a proper
    hereditary subset of $Q$ containing $P$. Thus $Q' = P$, forcing $Q \setminus P = \{q\}$.
    So by Lemma~\ref{lem:composition series} it suffices to show that if $P,Q$ are hereditary subsets of $I$, with $P\subset Q$, and $Q
    \setminus P= \{ q \}$  is a singleton, then $\Hh|_{W_Q \setminus W_P}$ is amenable.

    Fix such a $P, Q$ and $q$. Then,
    \begin{align*}
        W_Q \setminus W_P = \{\gamma\lambda \in \Hh^{(0)} : \gamma &\text{ is an initial segment}, r(\gamma)\in E_q^0\text{ and }\lambda\in E_q^{\infty}\}.
    \end{align*}
    The subset $X_q := \{\lambda \in \Hh^{(0)} : \lambda\text{ is an infinite path in }E_q\}$
    is a full open subset of the unit space of $\Hh|_{W_Q \setminus W_P}$. Hence, the final
    statement of \cite[Theorem~3.10]{FKPS} shows that $\Hh|_{W_Q \setminus W_P}$ is
    equivalent in the sense of \cite[Definition~2.1]{MRW} (see also
    \cite[Section~3]{Renault}) to the groupoid $\Hh|_{X_q}$. Since amenability is preserved
    by groupoid equivalence \cite[Theorem 2.2.17]{A-DR}, it therefore suffices to show that
    each $\Hh|_{X_q}$ is amenable.

    By definition, we have $\Hh|_{X_q} = \{(\gamma\lambda, (n_1, |\gamma|_\infty -
    |\nu|_\infty), \nu\lambda) : \gamma,\nu \in E_q^*, \lambda \in E_q^\infty, n_1 \in
    \ZZ^{(\infty)}\}$, and the map $c : (\gamma\lambda, (n_1, |\gamma|_\infty -
    |\nu|_\infty), \nu\lambda) \mapsto n_1$ is a continuous cocycle from $\Hh|_{X_q}$ to the
    countable abelian group $\ZZ^{(\infty)}$. Thus \cite[Proposition~9.3]{Spielberg} implies
    that $\Hh|_{X_q}$ is amenable provided that $c^{-1}(0)$ is amenable. We have $c^{-1}(0)
    \cong \{(\gamma\lambda, |\gamma|_\infty - |\nu|_\infty, \nu\lambda) : \gamma,\nu \in
    E_q^*, \lambda \in E_q^\infty\}$. If $q$ is a regular prime, this is precisely the graph
    groupoid $\Gg_{E_q}$, which is amenable by \cite[Corollary~5.5]{KPRR}; and if $q$ is a
    free prime, then $\gamma$ and $\nu$ are finite products of the $\alpha(q,j)$, and the map
    $(\gamma\lambda, |\gamma|_\infty - |\nu|_\infty, \nu\lambda) \mapsto |\gamma|_\infty -
    |\nu|_\infty$ is an isomorphism of $\Hh|_{X_q}$ onto $\ZZ^{k(q)}$, which again is amenable.
    So in either case $\Hh|_{X_q}$ is amenable, completing the proof.
\end{proof}

\section{The type semigroup}\label{Sect:TypeSemigroup}
\subsection{The type semigroup of a Boolean inverse semigroup}

Let $S$ be an inverse semigroup (always with $0$). We denote by $\mathcal{E}(S)$ the semilattice of idempotents of $S$.
We say that $x,y\in S$ are orthogonal, written $x\perp y$ if $x^*y= yx^*= 0$

Recall that a {\it Boolean inverse semigroup} is an inverse semigroup $S$ such that
$\mathcal{E}(S)$ is a Boolean ring (a ring with $x=x^2$ for all $x$), and such that every pair
$x,y \in S$ satisfying $x \perp y$ has a supremum, denoted $x\oplus y \in S$ (see
\cite[Definition~3-1.6]{Weh} for further details). For the following definition, we refer to
\cite{Weh} for ease of accessibility; but it did not originate there (see for example
\cite{Wallis, KLLR, LS}).

\begin{definition}
 \label{def:type-semigroup-BIS} \cite[page 98]{Weh}
Let $S$ be a Boolean inverse semigroup. The {\it type semigroup} (or type monoid) of $S$ is the
commutative monoid $\Typ (S)$ freely generated by elements $\typ (x)$, where $x\in
\mathcal{E}(S)$, subject to the relations
\begin{enumerate}
 \item $\typ (0) = 0$,
 \item $\typ (x) = \typ (y)$ whenever $x,y\in \mathcal{E}(S)$ and there is $s\in S$ such that $ss^*= x $ and $s^*s= y$.
\item $\typ (x\oplus y) = \typ (x) + \typ (y)$ whenever $x,y$ are orthogonal elements in $\mathcal{E}(S)$.
  \end{enumerate}
\end{definition}

By \cite[Corollary 4-1.4]{Weh}, $\Typ (S)$ is a conical refinement monoid.

\subsection{The type semigroup of an ample groupoid}

Let $\mathcal G$ be a (not-necessarily-Hausdorff) \'etale groupoid, with a Hausdorff locally
compact unit space $X:= \mathcal G^{(0)}$. Then the collection $S(\mathcal G)$ of all compact open
bisections of $\mathcal G$ forms a Boolean inverse semigroup. Indeed, it is well-known that it is
an inverse semigroup, and since the topology on $X$ is Hausdorff, one obtains that the semilattice
of idempotents of $S(\mathcal G)$, which coincides with the set of all the compact open subsets of
$X$, is a Boolean ring. Moreover, if we have two orthogonal compact open bisections, it is clear
that their disjoint union is again a compact open bisection.

\begin{definition}
 \label{def:typesemigroup-groupoid}
 Let $\mathcal G$ be a (non-necessarily Hausdorff) \'etale groupoid, with a Hausdorff locally compact base space $X:= \mathcal G^{(0)}$.
 We define the {\it type semigroup} of $\mathcal G$ as
 $$\Typ (\mathcal G) = \Typ (S(\mathcal G)).$$
 \end{definition}

If $\mathcal G$ is second countable, then the type semigroup $\Typ (\mathcal G)$ is a countable conical refinement monoid.

 Now, we will show that $\Typ (\mathcal G)$ coincides with some recently defined type semigroups \cite{BL,PSS,RS}.
To this end, we recall the definition in \cite{BL} (which is equivalent to the ones given in \cite{PSS} and \cite{RS}, as shown in \cite{RS}).

Let $\mathcal G$ be an ample Hausdorff groupoid. Define an equivalence relation $\sim $ on the set
$$ \left\{ \bigcup_{i=1}^n A_i \times \{ i \} \mid n\in \N, A_i\in \mathcal E(S(\mathcal G)) \right\},$$
as follows: for $A= \bigcup_{i=1}^n A_i\times \{ i \}$ and $B= \bigcup_{j=1}^m B_j \times \{ j \}$, write $A\sim B$ if there exist $l \in \N$,
open compact bisections $W_1,\dots , W_l$, and $n_1,\dots ,n_l,m_1,\dots , m_l\in \N$ such that
$$A= \bigsqcup _ {k=1}^l d(W_k) \times \{ n_k \} , \quad B= \bigsqcup_{k=1}^l r(W_k) \times \{ m_k \} .$$

The type semigroup of $\mathcal G$ is then defined in \cite{BL} to be
$$ S(\mathcal G, \mathcal G^a) = \left\{ \bigcup_{i=1}^n A_i \times \{ i \} \mid n\in \N, A_i\in \mathcal E(S(\mathcal G)) \right\} / \sim $$
with the operation
$$ \Big[ \bigcup_{i=1}^n A_i \times \{ i \}\Big]+\Big[ \bigcup_{j=1}^m B_j \times \{ j \}\Big] = \Big[ \bigcup_{i=1}^n A_i \times \{ i \} \cup  \bigcup_{j=1}^m B_j \times \{ j+n \}\Big].$$

\begin{proposition}
 \label{prop:equality-defs}
 Let $\mathcal G$ be an ample Hausdorff groupoid. Then $\Typ (\mathcal G ) \cong  S( \mathcal G, \mathcal G ^a)$.
 \end{proposition}

\begin{proof}
 There is a well-defined monoid homomorphism
 $$\Phi \colon \Typ (\mathcal G) \to S(\mathcal G, \mathcal G^a)$$
 defined by $\Phi ([A]) = [A\times \{ 1  \}]$, for every compact open subset $A$ of $\mathcal G^0$.
 Since $[A\times \{ 1 \} ]= [A \times \{ i \}]$ in $S( \mathcal G, \mathcal G^a)$, we easily see that this map is surjective.
 To show it is injective, assume that $\Phi ( \sum_{i=1}^n [A_i]) = \Phi ( \sum_{j=1}^m [B_j])$ for compact open subsets $A_i,B_j$ of $\mathcal G^0$.
 Then
 $$ A:= \bigcup_{i=1}^n A_i\times \{ i \} \sim B:=  \bigcup_{j=1}^m B_j \times \{ j \}, $$
 so there are open compact bisections $W_1,\dots , W_l$, and elements $n_1,\dots ,n_l,m_1,\dots , m_l\in \N$ such that
$$A= \bigsqcup _ {k=1}^l d(W_k) \times \{ n_k \} , \quad B= \bigsqcup_{k=1}^l r(W_k) \times \{ m_k \} .$$
 It follows that there is a partition $\{ 1,\dots , l \} = \bigsqcup_{i=1}^n I_i $ such that for each $i\in \{ 1\dots , n\}$ we have
 $A_i= \bigsqcup_{j\in I_i} d(W_j)$. We thus get in $\Typ (\mathcal G)$
 $$\sum _{i=1}^n [A_i] = \sum_{i=1}^n \sum_{j\in I_i} [d(W_j)] = \sum_{i=1}^n \sum_{j\in I_i} [r(W_j)] = \sum_{j=1}^l [r(W_j)] = \sum_{j=1}^m [B_j],$$
 showing injectivity.
 \end{proof}

Let us now move to the framework of these notes. Namely, let $(E,C)$ be an adaptable separated graph, and consider the ample Hausdorff groupoid $\mathcal G_{tight}(S(E,C))$.

We want to relate the type semigroup of $\mathcal G_{tight}(S(E,C))$ with our original primely
generated refinement monoid $M(E,C)$ (see Section~\ref{Section1}). To do this, we need the next
lemma. Recall that the notation $\mathcal Z (s)$, for $s\in S(E,C)$, has been introduced in
Definition~\ref{Def 7.1}, and that the family $\{ \mathcal Z (s) \}_{s\in S(E,C)}$ is a basis of
open compact bisections of $\mathcal H$. Since $\mathcal G_{tight}(S(E,C)) \cong \mathcal H$ with
an isomorphism sending $U_s$ onto $\mathcal Z(s)$, we will identify these two groupoids, and thus
consider $\mathcal Z(s)$ as open compact subsets of $\mathcal G_{tight}(S(E,C))$. Under this
identification, Definition~\ref{def:familycalP} and Definition~\ref{Def 7.1} agree on
$\mathcal{Z}(e)$ for any idempotent $e\in \mathcal E$.

\begin{lemma}
 \label{lem:form-of bisections} The following properties hold for the groupoid $\mathcal G_{tight}(S(E,C))$.
 \begin{enumerate}
  \item Suppose that $e,e_1,\dots ,e_N\in \mathcal E$. Then
  $$ \mathcal Z (e) = \bigsqcup_{i=1}^N \mathcal Z (e_i)$$
  if and only if $\{ e_1,\dots e_N\}$ is an orthogonal finite cover of $e$.
    \item Given any compact open bisection $U$, we have
  $$U = \bigsqcup_{i=1}^N \mathcal Z (s_i) ,$$
  for suitable $s_i \in S$, $i=1,\dots , N$.
  \end{enumerate}
\end{lemma}

\begin{proof} $\mbox{ }$
 (1) If $\{ e_1,\dots , e_N \}$ is an orthogonal finite cover of $e$, then $\mathcal Z (e) = \bigsqcup_{i=1}^N \mathcal Z (e_i)$ because the representation
 $\iota \colon S(E,C)\to \mathcal S_K(E,C)$ is tight. Conversely, assume that $\mathcal Z (e) = \bigsqcup_{i=1}^N \mathcal Z (e_i)$. Then the idempotents $e_i$
 are mutually orthogonal. Now let $f$ be a non-zero idempotent such that $f\le e$. Then $\mathcal Z (f)$ is a non-empty subset of $\mathcal Z (e)$, and thus
 $\mathcal Z (fe_i) =\mathcal Z (f) \cap \mathcal Z (e_i)$ is non-empty for some $i=1,\dots ,n$. This implies that $fe_i\ne 0$. Hence $\{ e_1,\dots e_N \}$ is
 an orthogonal finite cover of $e$.

 (2) Let $U$ be a compact open bisection. We have $U=\bigcup_{i=1}^n \mathcal Z (s_i)$ for elements
 $s_i\in S(E,C)$. Now, observe that $s(U)= \bigcup_{i=1}^n s(\mathcal Z(s_i)) = \bigcup_{i=1}^n \mathcal Z (s_i^ *s_i)$. By using the standard decomposition sketched in the proof
 of Lemma~\ref{lem:disjoint-unions}, we can write
 $$s(U) = \bigsqcup_{k=1}^N U_k ,$$
 a disjoint union of open compact subsets $U_k$ of $\widehat{\mathcal E}_{\infty}$ such that each $U_k$ is contained in a (not necessarily unique) set $\mathcal Z (s_i^*s_i)$.
 Now, by using Corollary~\ref{cor:compact-opens}, we can write each $U_k$ as a disjoint union of sets of the form $\mathcal Z (e_{k,l})$, for idempotents $e_{k,l}\in \mathcal E$.
 For each $k\in\{ 1,\dots , N \}$, let $\tau (k)\in \{ 1,\dots , n\}$ be such that $U_k\subseteq s(\mathcal Z (s_{\tau (k)}))$. Then
 $$ U = \bigsqcup_{k,l} \mathcal Z (s_{\tau (k)}e_{k,l}),$$
 because, since $U$ is a bisection, $s$ defines a homeomorphism $U \to s(U)$, sending $\mathcal Z (s_{\tau (k)}e_{k,l})$ onto $\mathcal Z (e_{k,l})$. \end{proof}

With this lemma, we are ready to show the main result of this paper.

\begin{theorem}
 \label{thm:monoids-agreement} Let $(E,C)$ be an adaptable separated graph, $S(E,C)$ be the inverse semigroup associated to $(E,C)$, and let $\mathcal G_{tight}(S(E,C))$ be the groupoid of germs associated to the canonical action of $S(E,C)$ on the space of ultrafilters $\widehat{\mathcal E}_{\infty}$. Then, there is a monoid isomorphism
 $$\psi \colon M= M(E,C) \to \Typ (\mathcal G_{tight}(S(E,C)))$$
 such that $\psi (a_v) = [\mathcal Z (v)]$ for every $v\in E^0$.
 \end{theorem}

\begin{proof}
 Since the defining relations in the semigroup $M(E,C)$ are generated by the expansions in Remark~\ref{rk:equalities}, Lemmas \ref{lem:expanded}~and~\ref{lem:form-of bisections} show that these
 relations are preserved under $\psi$. Hence, we obtain a monoid homomorphism $\psi$ satisfying $\psi (a_v) = [\mathcal Z (v)]$ for every $v\in E^0$.

Given an idempotent $e = \gamma \mathbf m (p) \gamma^*$ in $S(E,C)$, we obviously have $[\mathcal Z (e)] = [\mathcal Z (v)]$, where $v= v^p$ if $p$ is a free prime and $v= r(\lambda)\in E_p^0$ if
$p$ is a regular prime and $\mathbf m (p) = \lambda \lambda ^*$. It follows that the map $\psi$ is surjective.

We now show that $\psi$ is injective. Assume that $\psi (\sum_{i=1}^n a_{v_i} ) = \psi
(\sum_{j=1}^m a_{w_j})$ in $\Typ (\mathcal G)$. Using Proposition~\ref{prop:equality-defs}, we see
that this implies that
$$A:= \bigsqcup _{i=1}^n \mathcal Z (v_i) \times \{ i \} \sim B:= \bigsqcup_{j=1}^m  \mathcal Z (w_j) \times \{ j \}$$
in $S(\mathcal G, \mathcal G ^a)$. Now, by the definition of $\sim$ and Lemma~\ref{lem:form-of
bisections}, there exist $l\in \N$, $s_1,\dots , s_l \in S(E,C)$, and $n_1,\dots n_l,m_1,\dots
,m_l$ such that
$$A= \bigsqcup _{k=1}^l s(\mathcal Z (s_k))\times \{ n_k \},\quad B= \bigsqcup_{k=1}^l r(\mathcal Z (s_k) )\times \{ m_k \}.$$
It follows that there are partitions $\{ 1,\dots , l \} = \bigsqcup_{i=1}^n I_i$ and $\{ 1,\dots , l \} = \bigsqcup_{j=1}^m J_j$ such that
$$\mathcal Z (v_i) = \bigsqcup_{k\in I_i} s(\mathcal Z (s_k)) \qquad \text{and} \qquad \mathcal Z (w_j) = \bigsqcup _{k\in J_j} r(\mathcal Z (s_k)) $$
for $i=1,\dots , n$ and $j=1,\dots , m$. But now, $s(\mathcal Z (s_k)) = \mathcal Z (s_k^*s_k)$,
and $r(\mathcal Z (s_k)) = \mathcal Z (s_ks_k^*)$. Thus, $\mathcal Z (v_i) = \bigsqcup_{k\in I_i}
\mathcal Z (s_k^*s_k)$, which by Lemma~\ref{lem:form-of bisections}(1) implies that $\{ s_k^*s_k
\}_{k\in I_i}$ is an orthogonal finite cover of $v_i$, for $i=1,\dots , n$. It follows from
Proposition~\ref{prop:converseExpanded} that $\{ s_k^*s_k \}_{k\in I_i}$ is an expanded set of
$v_i$, and thus that $a_{v_i} = \sum _{k\in I_i} [s_k^*s_k]$ for $i=1,\dots , n$, where we denote
by $[s_k^*s_k]$ the class in $M(E,C)$ of the vertex in $E$ representing the idempotent $s_k^*s_k$.
Analogously, $a_{w_j} = \sum _{k\in J_j} [s_ks_k^*]$ for $j=1,\dots , m$. Since $[s_k^*s_k]=
[s_ks_k^*]$ in $M(E,C)$, we obtain
\begin{equation*}
 \sum_{i=1}^n a_{v_i}  = \sum_{i=1}^n \sum _{k\in I_i} [s_k^*s_k] = \sum_{k=1}^l [s_k^*s_k]
 = \sum_{k=1}^l [s_ks_k^*] = \sum_{j=1}^m \sum _{k\in J_j} [s_ks_k^*]
 = \sum_{j=1}^m a_{w_j},
\end{equation*}
and so $\psi$ is injective.
\end{proof}

The following corollary follows immediately from Theorems
\ref{thm:maingraphs-monoids}~and~\ref{thm:monoids-agreement}.

\begin{corollary}
 \label{cor:realizing-fgcrmr}
 Let $M$ be a finitely generated conical refinement monoid. Then there is an adaptable separated graph $(E,C)$ such that
 $$M\cong \Typ (\mathcal G_{tight} (S(E,C))).$$
 In particular, all finitely generated conical refinement monoids arise as type semigroups of ample Hausdorff groupoids.
 \end{corollary}

 \begin{remark}
  \label{rem:lastcomments}
 {\rm It would be interesting to know whether the natural map
 $$\iota \colon  M(E,C)\cong \Typ (\mathcal G_{tight} (S(E,C))) \to \mathcal V (A_K (\mathcal G_{tight} (S(E,C)))) $$
 is an isomorphism for any adaptable separated graph $(E,C)$. It is proven in \cite[Theorem A]{ABP19b} that the corresponding map $\iota_Q \colon M(E,C) \to \mathcal V (Q_K (E,C))$ is an isomorphism for any adaptable
 separated graph, where $Q_K(E)$ is a certain universal localization of
 $A_K(\mathcal G_{tight} (S(E,C)))$ (see \cite[Section 2]{ABP19b} for details). Note that the map $\iota_Q$ factors through the map $\iota$ and so, since $\iota_Q$ is injective, we conclude that at least the map $\iota $ is injective. In the C*-algebraic setting, it would be of interest to study the equivalent question; namely, whether the map
 $$\iota_{C^*}\colon M(E,C) \to  \mathcal V (C^* (\mathcal G_{tight} (S(E,C))))$$
 is injective. In this case, it is unlikely that the map $\iota _{C^*}$ is surjective.}
  \end{remark}

\section*{Acknowledgments}

This research project was initiated when the authors were at the Centre de Recerca Matem\`atica as part of the Intensive Research Program \emph{Operator
algebras: dynamics and interactions} in 2017, and the work was significantly supported by the
research environment and facilities provided there. We thank the Centre de Recerca Matem\`atica for its support.
We would like to thank to the anonymous referees for their very useful comments and suggestions.


\begin{thebibliography}{200}

\bibitem{AAS} \textsc{G. Abrams, P. Ara, M. Siles Molina}, Leavitt Path
Algebras, Springer Lecture Notes in Mathematics, vol. 2191, 2017.

\bibitem{A-DR} \textsc{C. Anantharaman-Delaroche, J. Renault}, Amenable groupoids, Monographies de L'Enseignement Math\'ematique \textbf{36}, \emph{L'Enseignement Math\'ematique}, Geneva, (2000).

\bibitem{Areal} \textsc{P.  Ara}, The realization problem for von Neumann regular rings,
in \emph{Ring Theory 2007. Proceedings of the Fifth China-Japan-Korea Conference}, (H. Marubayashi, K. Masaike, K. Oshiro, M. Sato, Eds.), Hackensack, NJ (2009) World Scientific, pp.~21--37.

\bibitem{ABP19} \textsc{P. Ara, J. Bosa, E. Pardo}, Refinement monoids and adaptable separated graphs, arXiv:1904.04255 [math.RA]. To appear in {\it Semigroup Forum}.

\bibitem{ABP19b} \textsc{P. Ara, J. Bosa, E. Pardo}, The realization problem for finitely generated refinement monoids,  arXiv:1907.03648v1[math.RA].

\bibitem{AE} \textsc{P. Ara, R.  Exel}, Dynamical systems associated to separated graphs, graph algebras, and paradoxical decompositions, \emph{Adv. Math.} \textbf{252} (2014), 748--804.

\bibitem{AG12} \textsc{P. Ara, K.R. Goodearl}, Leavitt path algebras of separated graphs,  \emph{J. Reine Angew. Math.} \textbf{669} (2012), 165--224.


\bibitem{AMFP} \textsc{P. Ara, M.A. Moreno, E. Pardo}, Nonstable K-Theory for graph algebras,
\emph{Algebra Rep. Th.} \textbf{10} (2007), 157-178.

\bibitem{AP16} \textsc{P. Ara, E. Pardo}, Primely generated refinement monoids, \emph{Israel J. Math.} \textbf{214} (2016), 379--419.

\bibitem{AP17} \textsc{P. Ara, E. Pardo}, Representing finitely generated refinement monoids as graph monoids \emph{J. Algebra} \textbf{480} (2017), 79--123.

\bibitem{APW08} \textsc{P Ara, F. Perera, F. Wehrung}, Finitely generated antisymmetric graph monoids, \emph{J. Algebra}~\textbf{320} (2008), 1963--1982.

\bibitem{BL} \textsc{C. B\"onicke, K. Li}, {Ideal structure and pure infiniteness of ample groupoid $C^*$-algebras}, \emph{Ergodic Theory and Dynamical Systems}, 1-30. doi:10.1017/etds.2018.39.

\bibitem{Brook} \textsc{G. Brookfield}, Cancellation in primely generated refinement monoids, \emph{Algebra Universalis} \textbf{46} (2001), 343--371.

\bibitem{BFG} \textsc{S. Bulman-Fleming, J. Fountain, V. Gould}, Inverse semigroups with zero:
    covers and their structure, \emph{J. Austral. Math. Soc. Ser. A} \textbf{67} (1999), 15--30.

\bibitem{CEP} \textsc{L.O. Clark, R. Exel, E. Pardo} A generalized uniqueness theorem and the graded ideal structure of Steinberg algebras, Forum Math \textbf{30} (2017), 533--552.

\bibitem{CFST} \textsc{L.O. Clark, C. Farthing, A. Sims, M. Tomforde}, A groupoid generalization of Leavitt path algebras, \emph{Semigroup Forum} \textbf{89} (2014), 501--517.

\bibitem{CH} \textsc{L.O. Clark, R. Hazrat}, \'Etale groupoids and Steinberg algebras, a concise introduction, arXiv:1901.01612v1 [math.RA].

\bibitem{ExelBraz} \textsc{R. Exel}, Inverse semigroups and combinatorial $C^*$-algebras, \emph{Bull. Braz. Math. Soc. (N.S.)} \textbf{39} (2008), 191--313.


\bibitem{EP1} \textsc{R. Exel, E. Pardo}, The tight groupoid of an inverse semigroup, \emph{Semigroup Forum} \textbf{92} (2016), 274--303.

\bibitem{FKPS} \textsc{C. Farsi, A. Kumjian, D. Pask, A. Sims}, Ample groupoids: equivalence, homology, and Matui's HK-conjecture,  \emph{arXix:1808.07807} [math.OA]. To appear in Münster J. Math.

\bibitem{directsum} \textsc{K.R.  Goodearl}, Von Neumann regular rings and direct sum
decomposition problems, in ``Abelian groups and modules" (Padova,1994), Math. and its Applics. 343, pp. 249--255, Kluwer Acad. Publ., Dordrecht (1995).


\bibitem{JL} \textsc{D.G. Jones, M.V. Lawson}, Graph inverse semigroups: their characterization and completion, \emph{J. Algebra} \textbf{409} (2014), 444--473.

\bibitem{KLLR} \textsc{G. Kudryavtseva, M.V. Lawson, D.H. Lenz, P. Resende}, Invariant means on
    Boolean inverse monoids, \emph{Semigroup Forum} \textbf{92} (2016), 77--101.

\bibitem{KPRR} \textsc{A. Kumjian, D. Pask, I. Raeburn, J. Renault}, Graphs, groupoids, and Cuntz-Krieger algebras, \emph{J. Funct. Anal.} \textbf{144} (1997), 505--541.

\bibitem{Lawson} \textsc{M.V. Lawson}, ``Inverse Semigroups, the Theory of Partial Symmetries'', \emph{World Scientific}, Singapore (1998).

\bibitem{Lawson2} \textsc{M. V. Lawson}, The structure of $0$-$E$-unitary inverse semigroups. I.
    The monoid case, \emph{Proc. Edinburgh Math. Soc. (2)} \textbf{42} (1999), 497--520.

\bibitem{LS} \textsc{M.V. Lawson, P. Scott}, {AF inverse monoids and the structure of countable
    MV-algebras}, \emph{J. Pure Appl. Algebra} \textbf{221} (2017), 45--74.

\bibitem{Matui} \textsc{H. Matui}, \'Etale groupoids arising from products of shifts of finite type, \emph{Adv. Math.}  \textbf{303} (2016), 502--548.

\bibitem{MS} \textsc{D. Milan, B. Steinberg}, On inverse semigroup $C^*$-algebras and crossed
    products, \emph{Groups Geom. Dyn.} \textbf{8} (2014), 485--512. 

\bibitem{MRW} \textsc{P.S. Muhly, J.N. Renault, D.P. Williams}, Equivalence and isomorphism for groupoid $C^*$-algebras, \emph{J. Operator Theory} \textbf{17} (1987), 3--22.

\bibitem{NO} \textsc{P. Nyland, E. Ortega}, Topological full groups of ample groupoids with applications to graph algebras, \emph{Internat. J. Math.} \textbf{30} (2019), no. 4, 1950018, 66 pp.

\bibitem{Ortega} \textsc{E. Ortega}, Homology of the Exel-Katsura-Pardo groupoid, arXiv:1806.09297v2 [math.OA].

\bibitem{PSS} \textsc{D. Pask, A. Sierakowski, A. Sims}, Unbounded quasitraces, stably finiteness and pure infiniteness, \emph{arXiv:1705.01268} [math.OA].

\bibitem{RS} \textsc{T. Rainone, A. Sims}, {A dichotomy for groupoid $C^*$-algebras}, \emph{arXiv:1707.04516} [math.OA].

\bibitem{Renault} \textsc{J. Renault}, A groupoid approach to $C^*$-algebras,  \emph{Lecture Notes in Mathematics} \textbf{793}, Springer, Berlin, 1980.



\bibitem{Rigby} \textsc{S.W.  Rigby}, The groupoid approach to Leavitt path algebras, arXiv:1811.02269v2 [math.RA].To appear in Proceedings of the International Workshop on Leavitt Path Algebras and K-Theory.

\bibitem{Scarparo} \textsc{E. Scarparo}, Homology of odometers, \emph{Ergodic Theory and Dynamical
    Systems}, published online 13 March 2019,
    \href{https://doi.org/10.1017/etds.2019.13}{DOI:10.1017/etds.2019.13}.

\bibitem{SimsCRMNotes}  \textsc{A. Sims}, Hausdorff \'etale groupoids and their $C^*$-algebras, to
    appear in \emph{Advanced Courses in Mathematics}, \emph{Operator algebras and dynamics: groupoids, crossed products and Rokhlin dimension}, Birkh\"auser, CRM Barcelona, 2018.

\bibitem{Spielberg} \textsc{J. Spielberg}, Groupoids and $C^*$-algebras for categories of paths, \emph{Trans. Amer. Math. Soc.} \textbf{366} (2014), 5771--5819.

\bibitem{Steinberg} \textsc{B. Steinberg}, A groupoid approach to discrete inverse semigroup algebras, \emph{Adv. Math.} \textbf{223} (2010), 689--727.

\bibitem{Steinb16}  \textsc{B. Steinberg}, Simplicity, primitivity and semiprimitivity of \'etale groupoid algebras with applications to inverse semigroup algebras, \emph{J. Pure Appl. Algebra} \textbf{220} (2016), 1035--1054.



\bibitem{Wallis}  \textsc{Alistair R. Wallis}, \emph{Semigroup and category-theoretic approaches
    to partial symmetry}, Ph.D. thesis, Heriot-Watt University, Edinburgh, 2013.


\bibitem{Weh} \textsc{F. Wehrung}, {Refinement monoids, equidecomposable types, and Boolean inverse semigroups}, \emph{Springer Lecture Notes in Math.} 2188 (2017).

\bibitem{WilliamsGpds} \textsc{D.P. Williams}, A Tool Kit for Groupoid $C^*$-algebras (in preparation).

\end{thebibliography}
\end{document}